\documentclass[11pt]{article}
	\usepackage{caption}
    \usepackage{url}
    \usepackage{verbatim}
	\usepackage{amsmath,bm}
    \usepackage{color}
	\usepackage{amsthm}
	\usepackage{amssymb}
	\usepackage{amsfonts}
	\usepackage{graphicx}
	\graphicspath{ {figures/}{figures/eps/}{figures/pdf/}{figures/png/}{figures/eps/} }
	\usepackage{epstopdf} 
	\usepackage{nicefrac}
	\usepackage{algorithm}
	\usepackage{algpseudocode}
	\usepackage{float}
	\usepackage{subfig}
	\def\captionfont{\setb@se{11pt}\protect\footnotesize}
    \def\captionfont{\protect\footnotesize}
    \newcommand{\iprd}[2]{\left( #1 , #2 \right)}

    \newcommand{\aiprd}[2]{a\!\left( #1 , #2 \right)}

    \newcommand{\Soh}{\mathring{S}_h}
    \newcommand{\bI}{{\bf I}}

	\def\norm#1#2{\left\| #1 \right\|_{#2}}

	\newcommand{\avephio}{\overline{\phi}_0}
	\newcommand{\cirphih}{\mathring{\phi}_h}
	\newcommand{\avemuh}{\overline{\mu}_h}
	\newcommand{\newtmuh}{\mathring{\mu}_{h,j}}
	\newcommand{\newtphih}{\mathring{\phi}_{h,j}}
	\newcommand{\cirmuh}{\mathring{\mu}_h}

	\newcommand{\dtau}{\delta_\tau}
	
	\newcommand{\phih}{\phi_h}

	\newcommand{\bv}{{\bf v}}

	\newcommand{\muh}{\mu_h}

	\newcommand{\bP}{{\bf P}}
	\newcommand{\bK}{{\bf K}}
	\newcommand{\bM}{{\bf M}}
	\newcommand{\bc}{{\bf c}}

	\newtheorem{thm}{Theorem}[section]

	\newtheorem{rem}[thm]{Remark}
	
\def\tK{\tilde{\bf K}}
\def\bI{{\bf I}}
\def\bB{{\bf B}}
\def\bC{{\bf C}}
\def\bv{{\bf v}}
\def\bV{{\bf V}}
\def\d{\displaystyle}
\begin{document}
\title{A Robust Solver for a Mixed Finite Element Method for the
 Cahn-Hilliard Equation\thanks{The work of the first and third authors
 was supported in part by the National Science
  Foundation under Grant No. DMS-16-20273.}}
	\author{
Susanne C. Brenner\thanks{Department of Mathematics and Center for Computation \& Technology,
 Louisiana State University, Baton Rouge, LA 70803 (brenner@math.lsu.edu)},
	\and
Amanda E. Diegel\thanks{Department of Mathematics  and Center for Computation \& Technology,
 Louisiana State University, Baton Rouge, LA 70803 (diegel@math.lsu.edu)},
	\and
Li-Yeng Sung\thanks{Department of Mathematics and Center for Computation \& Technology,
 Louisiana State University, Baton Rouge, LA 70803 (sung@math.lsu.edu)}}

\date{}
\maketitle
	
\numberwithin{equation}{section}
	
\begin{abstract}\noindent
 We develop a robust solver for a mixed finite element convex splitting scheme
 for the Cahn-Hilliard equation.  The key ingredient of the solver is a preconditioned
 minimal residual algorithm (with a multigrid preconditioner) whose performance
 is independent of the spacial mesh size and the
  time step size for a given interfacial width parameter.  The dependence on the interfacial
  width parameter is also mild.
\end{abstract}
\par\quad
\begin{minipage}{5.8in}
\end{minipage}
%
\section{Introduction}\label{sec:Introduction}
 Let $\Omega\subset \mathbb{R}^d$, $d=2,3$, be an open polygonal or polyhedral domain,
  and consider
 the following form of the Cahn-Hilliard energy \cite{cahn61}:
\begin{align}
 E(\phi) = \int_{\Omega} \left(\frac{1}{4 \varepsilon} (\phi^2-1)^2 +
 \frac{\varepsilon}{2} | \nabla \phi |^2\right) d{x},
\label{eq:ch-energy}
\end{align}
 where  $\varepsilon >0$ is a constant, and $\phi \in H^1(\Omega)$ represents a
 concentration field. The phase equilibria are represented by $\phi = \pm 1$ and
 the parameter $\varepsilon$ represents a non-dimensional interfacial
 width between the two phases.
\par
 The Cahn-Hilliard equation, which can be interpreted as the gradient flow of the energy
 \eqref{eq:ch-energy} in the dual space of $H^1(\Omega)$,
 is often represented in mixed form by
\begin{subequations}\label{subeqs:CH}
	\begin{eqnarray}
 \partial_t \phi = \varepsilon\Delta \mu,  && \text{in} \,\Omega ,
	\label{eq:CH-mixed-a-alt}
	\\
 \mu = \varepsilon^{-1}\, \left(\phi^3 - \phi\right) -
 \varepsilon \Delta \phi, && \text{in} \,\Omega,
	\label{eq:CH-mixed-b-alt}
	\end{eqnarray}
\end{subequations}
 together with the boundary conditions $\partial_n \phi =0$
 and $\partial_n \mu = 0$.
\par
 Let $T$ be a positive number and  $H^{-1}_N(\Omega)$ be the dual space of $H^1(\Omega)$.
 A weak formulation of \eqref{eq:CH-mixed-a-alt}--\eqref{eq:CH-mixed-b-alt} is to
  find $(\phi,\mu)$ such that
\begin{subequations}\label{subeqs:Spaces}
	\begin{eqnarray}
\phi &\in& L^\infty\left(0,T;H^1(\Omega)\right)\cap
 L^4\left(0,T;L^\infty(\Omega)\right),	\\
\partial_t \phi &\in&  L^2\bigl(0,T; H_N^{-1}(\Omega)\bigr),
	\\
\mu &\in& L^2\bigl(0,T;H^1(\Omega)\bigr),
	\end{eqnarray}
\end{subequations}
 and, for almost all $t\in (0,T)$,
\begin{subequations}\label{subeqs:Equations}
	\begin{align}
\langle \partial_t \phi ,\nu \rangle + \varepsilon \,\aiprd{\mu}{\nu}
 &= 0  \quad \forall \,\nu \in H^1(\Omega),
	\label{eq:weak-ch-a}\\
\iprd{\mu}{\psi}-\varepsilon \,\aiprd{\phi}{\psi} -
 \varepsilon^{-1}\iprd{\phi^3-\phi}{\psi} &= 0  \quad \forall \,\psi\in H^1(\Omega).
	\label{eq:weak-ch-b}
	\end{align}
\end{subequations}
 Here $\langle \cdot, \cdot \rangle$ denotes the duality pairing between the spaces
 $H^{-1}_N(\Omega)$ and $H^1(\Omega)$, $(\cdot,\cdot)$ is the inner product of $L^2(\Omega)$,
  and
\begin{equation*}
  a({u},{v})=(\nabla u,\nabla v).
 \end{equation*}
 The proof for the existence and uniqueness of the weak solution for \eqref{subeqs:Spaces}--\eqref{subeqs:Equations}
 with initial data
\begin{equation}\label{eq:InitialData}
 \phi(0)=\phi_0\in \,H^2_N(\Omega)=\{v\in H^2(\Omega):\,\partial v/\partial n=0
  \; \text{on}\; \partial\Omega\}
\end{equation}
 can be found for example in \cite{Temam:1988:DynamicalSystem}.
\par
 The Cahn-Hilliard energy \eqref{eq:ch-energy} along with the system \eqref{subeqs:CH}
  was originally developed
 to model phase separation of a binary fluid \cite{cahn61,cahn58,elliott86}.
 However, variations of the Cahn-Hilliard system are quickly becoming one
 of the most popular components in what are known as phase field models.
 The role that the Cahn-Hilliard equation takes in these models may best
 be described as creating an indicator function so that explicit tracking
 of the interface between two phases is not required.
 The growing number of applications include two phase flows, Hele-Shaw flows, copolymer fluids,
 crystal growth, void electromigration, vesicle membranes and more (cf.
 \cite{feng06,lee02a,choksi11,van09,barrett05,du07} and the references therein).
\par
 There is a vast literature on numerical methods for the Cahn-Hilliard equation
 (cf. \cite{tierra15,diegel16,lee14b,wang16,zhou15} and the
 references therein) and solvers  based on various numerical schemes
 were developed in
 \cite{axelsson13,boyanova12,ceniceros07,kay06,kim07,kim04,
 stogner08,wise10,shin13,lee14,wang16}.
 We will consider
 the mixed finite element method for
 \eqref{subeqs:Spaces}--\eqref{eq:InitialData} investigated in
 \cite{diegel15}.  It is
 based on the convex splitting scheme in time \cite{eyre98} given by
\begin{subequations}\label{subeqs:ConvexSplitting}
\begin{align}
 \frac{\phi^{m} - \phi^{m-1} }{\tau}= \varepsilon \Delta \mu^m,\\
 \Delta \mu^m = \frac{1}{\varepsilon} \left((\phi^m)^3 - \phi^{m-1}\right)
 - \varepsilon \Delta \phi^m,
\end{align}
\end{subequations}
  where $\tau$ is the size of the time step, and a spatial discretization that employs
   Lagrange finite elements.  This mixed finite element method is unconditionally stable and
   has optimal convergence in both time and space.
    Our goal is to develop a robust solver for this mixed finite element method.
\par
 The remainder of this paper is organized as follows.  The mixed finite element method
 is introduced in Section~\ref{sec:MFEM}, followed by the construction
 and analysis of the  solver
 in Section \ref{sec:Solver}.
 Numerical results that demonstrate the performance of the
 solver are presented in Section \ref{sec:Numerics}, and we end
 the paper with some concluding remarks in Section~\ref{sec:Conclusions}.
%
\section{A Mixed Finite Element Method}\label{sec:MFEM}
 Let $M$ be a positive integer, $0=t_0 < t_1 < \cdots < t_M = T$
  be a uniform partition of $[0,T]$ and
 $\mathcal{T}_h$ be a quasi-uniform family of triangulations of $\Omega$
 (cf. \cite{BScott:2008:FEM}).
 The Lagrange finite element space $S_h\subset H^1(\Omega)$ is given by
\begin{equation*}
  S_h=\{v\in C(\bar\Omega):\, v|_K \in {\mathcal P}_1(K)\,
  \forall \,\,  K\in \mathcal{T}_h \},
\end{equation*}
 and we define
\begin{equation*}
\Soh = S_h\cap L_0^2(\Omega),
\end{equation*}
 where $L_0^2(\Omega)$ is the space of square integrable functions with zero mean.
\par
 The mixed finite element scheme for \eqref{subeqs:ConvexSplitting}
 investigated in \cite{diegel15}
 is defined as follows:  For $1\le m\le M$, find $\phih^m,\muh^m\in S_h$  such that
\begin{subequations}\label{subeqs:Schemea}
	\begin{align}
 \iprd{\dtau \phih^m}{\nu} + \varepsilon \,\aiprd{\muh^m}{\nu}
    &=  0  \, & \forall \, \nu \in S_h ,
	\label{eq:ch-scheme-a}
	\\
 \iprd{\muh^m}{\psi}-\varepsilon^{-1} \,\big({\left(\phih^m\right)^3
 -\phih^{m-1}},{\psi}\big)
  -\varepsilon \,\aiprd{\phih^m}{\psi}  &=
  0  \,&\forall \, \psi\in S_h ,
	\label{eq:ch-scheme-b}\\
	\phih^0 -R_h \phi_0&=0.\label{eq:ch-scheme-c}
	\end{align}
	\end{subequations}
 Here
	\begin{equation*}
\dtau \phih^{m} = \frac{\phih^m-\phih^{m-1}}{\tau},
	\end{equation*}
 where $\tau=T/M$ is the size for the time step, and the Ritz projection operator $R_h:H^1(\Omega)\longrightarrow S_h$ is
 defined by
\begin{subequations}\label{subeqs:Rh}
\begin{align}
  a(R_hv,w)&=(v,w) \qquad \forall\,w\in S_h, \\
  (R_h v-v,1)&=0.  \label{eq:RhMean}
\end{align}
\end{subequations}
\begin{rem}\label{rem:Properties}
  The energy law
\begin{align*}
 E(\phi(t)) + \int_0^t \varepsilon \norm{\nabla \mu(s)}{L^2(\Omega)}^2 ds = E(\phi(0))
 \qquad\forall\,t\in[0,T]
\end{align*}
 is a key property of the solution of \eqref{subeqs:Spaces}--\eqref{eq:InitialData}.
 It can be shown {\rm\cite{diegel15}}
  that the solution of the finite element method defined by
 \eqref{subeqs:Schemea}--\eqref{subeqs:Rh} also satisfies a
 similar energy law, which leads to
  $\phih \in L^{\infty}(0,T;L^{\infty}(\Omega))$ and
  $\muh \in L^{2}(0,T;L^{\infty}(\Omega))$.
Moreover, under the assumption that
$\phi \in H^2(0,T;L^2(\Omega)) \cap L^{\infty}(0,T;W^{1,6}(\Omega)) \cap H^1(0,T;H^2(\Omega)), \mu \in L^{\infty}(0,T;H^1(\Omega)) \cap L^2(0,T;H^2(\Omega))$,
and $ 0 \le \tau \le \tau_0$ for a sufficiently small $\tau_0$, the error estimate
\begin{align}\label{eq:ErrorEstimate}
\max\limits_{1 \le m \le M} \norm{\nabla \phi^m - \nabla \phih^m}{L^2}^2
+ \tau \sum_{m=1}^M \norm{\nabla \mu^m - \nabla \muh^m}{L^2}^2
\le C(\varepsilon,T) (\tau^2 + h^2)
\end{align}
 holds for a positive constant $C$ that depends on $\varepsilon$ and $T$
 but does not depend on $\tau$ and $h$.
\end{rem}
\par
 It follows from \eqref{subeqs:Schemea} that
     $(\phih^m,1)=(\phi_0,1)$ for $0\leq m\leq M$,
 and hence
\begin{equation}\label{eq:ave-phi-def}
\phih^m = \avephio + \cirphih^m \quad \text{for} \quad 0\leq m \leq M,
\end{equation}
 where $\avephio = {\iprd{\phi_0}{1}}/{\iprd{1}{1}}$ is the mean of
 $\phi_0$ over $\Omega$ and $\cirphih \in \Soh$.
 We can also write
\begin{equation}\label{eq:ave-mu-def}
\muh^m = \avemuh^m + \cirmuh^m,
\end{equation}
 where $\avemuh^m$ is a constant function and $\cirmuh^m \in \Soh$.
\par
 Using \eqref{eq:ave-phi-def} and \eqref{eq:ave-mu-def}, we can rewrite
 \eqref{eq:ch-scheme-a}--\eqref{eq:ch-scheme-b}  in the following equivalent form:
  For $1\le m\le M$, find $\cirphih^m,\cirmuh^m\in \Soh$ such that
\begin{subequations}\label{subeqs:Schemeb}
\begin{align}
\big({\dtau \cirphih^m},{\nu} \big)+ \varepsilon \,\aiprd{\cirmuh^m}{\nu} =
\,& 0  \, & \forall \, \nu \in \Soh ,
	\label{eq:chmz-scheme-a}
	\\
({\cirmuh^m},{\psi})-\varepsilon^{-1} \,\big({(\cirphih^m + \avephio)^3
 -\cirphih^{m-1}},{\psi} \big)
 -\varepsilon \,a({\cirphih^m},{\psi})= \,& 0  \,&\forall \, \psi\in \Soh ,
	\label{eq:chmz-scheme-b}
\end{align}
\end{subequations}
 where
\begin{equation*}
\dtau \cirphih^{m} = \frac{\cirphih^m-\cirphih^{m-1}}{\tau} .
\end{equation*}
\par
 Note that  we can recover the constant function $\avemuh^m$  from $\cirphih^m$ and $\cirphih^{m-1}$
  through the relation
$$\iprd{\avemuh^m}{1} = \varepsilon^{-1} \big((\cirphih^m + \avephio)^3 -\cirphih^{m-1},{1}\big)$$
 that follows from \eqref{eq:ch-scheme-b}, \eqref{eq:ave-phi-def} and \eqref{eq:ave-mu-def}.
\begin{rem}\label{rem:unique-solvability}
 The nonlinear system \eqref{subeqs:Schemeb}
 defines the first order optimality condition
 at a minimum $\cirphih^m$ of the  convex functional
 $\Phi:\Soh\longrightarrow \mathbb{R}$ given by
\begin{align*}
 \Phi_h(\varphi) = \frac{\varepsilon}{2\tau} \aiprd{\varrho}{\varrho}
 + \frac{1}{4\varepsilon} \norm{\varphi + \avephio}{L^4}^4 - \frac{1}{\varepsilon} \iprd{\cirphih^{m-1}}{\varphi}
  + \frac{\varepsilon}{2} \aiprd{\varphi}{\varphi},
\end{align*}
 where $\varrho \in \Soh$ is defined by
\begin{align*}
\varepsilon \, \aiprd{\varrho}{\nu} + \iprd{\varphi - \cirphih^{m-1}}{\nu} = 0
\quad \forall \,\nu \in \Soh.
\end{align*}
 Since this convex minimization problem has a unique minimum by the
 standard theory {\rm \cite{ET:1999:Convex}}, the system
 \eqref{subeqs:Schemeb} $($and hence \eqref{subeqs:Schemea}$)$ is also uniquely solvable.
\end{rem}
\section{A Robust Solver}\label{sec:Solver}
 We will solve the nonlinear system \eqref{subeqs:Schemeb} by Newton's iteration.  Let
 $(\newtphih^m,\newtmuh^m)\in \Soh\times\Soh$ be the output of the $j$-th step.
    In order to advance the iteration, we need to find
 $(\delta_j \mathring{\mu}, \delta_j \mathring{\phi}) \in \Soh \times \Soh$ such that
\begin{subequations}\label{subeqs:Jacobian1}
\begin{alignat}{3}
\tau  \varepsilon \, a({\delta_j \mathring{\mu}},{\nu}) + (\nu,\delta_j \mathring{\phi})
 &= F_j(\nu) &\quad& \forall\, \nu \in \Soh,\\
 ({\delta_j \mathring{\mu}},{\psi}) - \left[3 \varepsilon^{-1}
 \big({(\phi_{h,j}^m)^2 \delta_j \mathring{\phi}},{\psi}\big)
 + \varepsilon \,a({\delta_j \mathring{\phi}},{\psi})\right]
 &= G_j(\psi) &\quad& \forall\, \phi \in \Soh,
\end{alignat}
\end{subequations}
where $\phi_{h,j}^m = \mathring{\phi}_{h,j}^m + \avephio$ and
\begin{subequations}\label{subeqs:Jacobian2}
\begin{align}
 F_j(\nu) &= \tau \varepsilon \, a({\mathring{\mu}_{h,j}^m},{\nu} )
+ ({\mathring{\phi}_{h,j}^m - \mathring{\phi}_{h,j}^{m-1}},{\nu}),
\\
G_j(\psi) &= ({\mathring{\mu}_{h,j}^m},{\psi})-
\left[ \varepsilon^{-1} \big({(\phi_{h,j}^m)^3 -
\mathring{\phi}_{h,j}^{m-1}},{\psi}\big)
 + \varepsilon \, a({\mathring{\phi}_{h,j}^{m}},{\psi}) \right].
\end{align}
\end{subequations}
 The next output of the Newton iteration is then given by
\begin{align}\label{eq:NextOutput}
  (\mathring{\mu}_{h,j+1}^m,\mathring{\phi}_{h,j+1}^m)
= (\mathring{\mu}_{h,j}^m,\mathring{\phi}_{h,j}^m) -
       (\delta_j \mathring{\mu}, \delta_j \mathring{\phi}).
\end{align}
 Below we will construct a robust solver for \eqref{subeqs:Jacobian1}.
\par
 First  we circumvent the inconvenient zero mean constraint by reformulating
 \eqref{subeqs:Jacobian1} as the following equivalent problem:
 Find $(\delta_j \mu, \delta_j \phi ) \in S_h \times S_h$ such that
\begin{subequations}\label{subeqs:NewJacobian1}
\begin{alignat}{3}
\tau \varepsilon \big[ \aiprd{\delta_j \mu}{\nu} +
\iprd{\delta_j \mu}{1} \iprd{\nu}{1} \big]
+ \iprd{\nu}{\delta_j \phi} &=
\tilde{F}_j(\nu) &\quad& \forall\, \nu \in S_h,\\
\iprd{\delta_j \mu}{\psi} - \left[3 \varepsilon^{-1}
\big({(\phi^m_{h,j})^2 \delta_j \phi},{\psi}\big) +
\varepsilon \big[a({\delta_j \phi},{\psi}) +
\iprd{\delta_j \phi}{1} \iprd{\psi}{1} ] \right]
 &= \tilde{G}_j(\psi) &\quad& \forall\, \psi \in S_h,
\end{alignat}
\end{subequations}
 where
\begin{equation}\label{eq:NewJacobian2}
\tilde{F}_j(\nu) =
\begin{cases}
F_j(\nu) & \text{if } \nu \in \Soh \\
0 & \text{if } \nu=1
\end{cases}
\quad\text{and}\quad
\tilde{G}_j(\psi) =
\begin{cases}
G_j(\psi) & \text{if } \psi \in \Soh \\
0 & \text{if } \psi=1
\end{cases}.
\end{equation}
\par
\begin{rem}\label{rem:Equivalence}
 It is easy to check that both \eqref{subeqs:Jacobian1} and
 \eqref{subeqs:NewJacobian1} are well-posed
 linear systems and that the solution
 $(\delta_j \mathring{\mu}, \delta_j \mathring{\phi})$ of
 \eqref{subeqs:Jacobian1}--\eqref{subeqs:Jacobian2} also satisfies
 \eqref{subeqs:NewJacobian1}--\eqref{eq:NewJacobian2}.
\end{rem}
\par
 Under the change of variables
\begin{equation*}
\iprd{\delta_j \mu}{\nu}  \rightarrow \tau^{-\frac{1}{4}}
\varepsilon^{-\frac{1}{2}} \iprd{\delta_j \mu}{\nu}
 \quad\text{and} \quad
 \iprd{\delta_j \phi}{\psi}  \rightarrow \tau^{\frac{1}{4}}
 \varepsilon^{\frac{1}{2}} \iprd{\delta_j \phi}{\psi},
\end{equation*}
 the system \eqref{subeqs:NewJacobian1} becomes
\begin{subequations}\label{subeqs:CV}
\begin{align}
\tau^{\frac{1}{2}} \left[ \aiprd{\delta_j \mu}{\nu} +
 ({\delta_j \mu},{1}) ({\nu},{1}) \right]
+ ({\nu},{\delta_j \phi}) &= \tilde{F}_j(\tau^{-\frac{1}{4}}
 \varepsilon^{-\frac{1}{2}} \nu),\\
 ({\delta_j \mu},{\psi}) - \big[3 \tau^{\frac{1}{2}}
 \big({ (\phi^m_{h,j})^2 \delta_j \phi},{\psi} \big)
+ \tau^{\frac{1}{2}}  \varepsilon^2 \left[ a({\delta_j \phi},{\psi})
 + ({\delta_j \phi},{1} )({\psi},{1}) \right] \big]
&= \tilde{G}_j(\tau^{\frac{1}{4}} \varepsilon^{\frac{1}{2}}\psi),
\end{align}
	\label{eq:unconstrained-final}
\end{subequations}
 for all $(\nu,\psi)\in S_h\times S_h$.
\par
 Let $n_h$ be the dimension of $S_h$ and
 $\varphi_1,\ldots,\varphi_{n_h}$  be the standard nodal basis (hat) functions for $S_h$.
 The system matrix for \eqref{subeqs:CV} is given by
\begin{equation}\label{eq:SystemMatrix}
\begin{bmatrix}
\tau^\frac{1}{2} \left(\bK + \bc \bc^t\right) & \bM \\
\bM & -\tau^\frac{1}{2} {\bf J} (\phi_{h,j}^m) -
\tau^\frac{1}{2}\varepsilon^2 \left(\bK + \bc \bc^t\right)
\end{bmatrix},
\end{equation}
 where the stiffness matrix $\bf K$ is  defined by
 ${\bf K}(k,\ell)=(\nabla\varphi_k,\nabla\varphi_\ell)$,
 the mass matrix $\bf M$ is  defined by
  ${\bf M}(k,\ell)=(\varphi_k,\varphi_\ell)$, the vector $\bf c$ is  defined by
 ${\bf c}(k)=(\varphi_k,1)$, and the matrix ${\bf J} (\phi_{h,j}^m)$ is defined by
   $${\bf J} (\phi_{h,j}^m)(k,\ell)=3\big((\phi^m_{h,j})^2\varphi_k,\varphi_\ell\big).$$
\par
 Note that, since the mixed finite element method is convergent, we can expect
 $(\phi^m_{h,j})^2$ to be close to 1 away from an interfacial region with
  width $\varepsilon$.
 Therefore, for small $\varepsilon$, we can take $(\phi^m_{h,j})^2$ to be 1
 in the system matrix, i.e., we can
  replace ${\bf J} (\phi_{h,j}^m)$ by $3{\bf M}$ in \eqref{eq:SystemMatrix}.
  The following result is motivated
  by this observation.
\begin{thm}\label{thm:Preconditioner}
 Let the matrices $\bf B$ and $\bf P$ be defined by
\begin{align}
{\bf B}&=\begin{bmatrix}
\tau^\frac{1}{2} (\bK + \bc \bc^t) & \bM \\
\bM & -3\tau^\frac{1}{2} \bM - \tau^\frac{1}{2}\varepsilon^2 (\bK + \bc \bc^t)
\end{bmatrix},\label{eq:block-system-simple}\\
\bP &=
 \begin{bmatrix}
\tau^\frac{1}{2}( \bK+\bc\bc^t) + \bM & {\bf 0} \\
{\bf 0} & \tau^\frac{1}{2} \varepsilon^2 (\bK+\bc\bc^t)+ \bM
\end{bmatrix},
\label{eq:preconditioner}
\end{align}
 where $0\leq \tau,\varepsilon\leq 1$.
 There exist two positive constants $C_1$ and $C_2$ independent of
 $\varepsilon$, $h$ and $\tau$ such that
\begin{equation}\label{eq:eigenvalue-condition}
C_2 \max(\tau^\frac{1}{2},\varepsilon) \le|\lambda | \le C_1 \quad
\text{\rm for any eigenvalue } \lambda \text{ of } \bP^{-1} {\bf B}.
\end{equation}
\end{thm}
\begin{proof} A simple calculation shows that
\begin{align*}
 \bP^{-1}\bB&=\left(\begin{bmatrix} \bM & 0\\ 0&\bM\end{bmatrix}
     \begin{bmatrix} \tau^\frac12 \tK +\bI& 0 \\
      0 &\tau^\frac12\varepsilon^2\tK+\bI
 \end{bmatrix}\right)^{-1}
 \left(\begin{bmatrix} \bM & 0\\ 0&\bM\end{bmatrix}
\begin{bmatrix}
   \tau^\frac12\tK &\bI\\
   \bI&-3\tau^\frac12\bI-\tau^\frac12\varepsilon^2\tK
 \end{bmatrix}\right)\\
   &=\begin{bmatrix} \tau^\frac12 \tK +\bI& 0 \\
      0 &\tau^\frac12\varepsilon^2\tK+\bI
 \end{bmatrix}^{-1}
 \begin{bmatrix}
   \tau^\frac12\tK &\bI\\
   \bI&-3\tau^\frac12\bI-\tau^\frac12\varepsilon^2\tK
 \end{bmatrix},
\end{align*}
 where $\tK=\bM^{-1}(\bK+\bc\bc^t)$ and $\bI$ is the $n_h\times n_h$ identity matrix.
\par
 By the spectral theorem, there exist $\bv_1,\ldots,\bv_{n_h}\in\mathbb{R}^{n_h}$ and
 positive numbers $\kappa_1,\ldots,\kappa_{n_h}$ such that
\begin{equation*}
  \tK \bv_j=\kappa_j \bv_j \qquad\text{for}\quad 1\leq j\leq n_h
\end{equation*}
 and
\begin{equation*}
     \bv_j^t\bM\bv_\ell=\begin{cases} 1 &\quad \text{if $j=\ell$}\\
           0 &\quad\text{if $j\neq\ell$} \end{cases} \;.
\end{equation*}
\par
 Observe that the two dimensional space  $\bV_j$ spanned by
\begin{equation*}
  \begin{bmatrix} \bv_j \\ 0 \end{bmatrix} \quad \text{and} \quad
  \begin{bmatrix} 0 \\ \bv_j\end{bmatrix}
\end{equation*}
 is invariant under $\bP^{-1}\bB$ and
\begin{equation*}
 \bP^{-1}\bB\left(\alpha\begin{bmatrix} \bv_j \\ 0 \end{bmatrix} +\beta
     \begin{bmatrix} 0 \\ \bv_j\end{bmatrix}\right)=\gamma
     \begin{bmatrix} \bv_j \\ 0 \end{bmatrix} +\delta
     \begin{bmatrix} 0 \\ \bv_j\end{bmatrix},
\end{equation*}
 where
\begin{equation*}
  \begin{bmatrix} \gamma\\ \delta \end{bmatrix}=
  \begin{bmatrix} \tau^\frac12\kappa_j+1&0\\ 0&\tau^\frac12\varepsilon^2\kappa_j+1
  \end{bmatrix}^{-1}
  \begin{bmatrix} \tau^\frac12\kappa_j & 1\\ 1& -3\tau^\frac12
  -\tau^\frac12\varepsilon^2\kappa_j \end{bmatrix}
  \begin{bmatrix}\alpha \\ \beta \end{bmatrix}.
\end{equation*}
\par
 It follows that the eigenvalues of $\bP^{-1}\bB$ are precisely the eigenvalues of the matrix
\begin{align*}
\bC_j&=\begin{bmatrix}
   \tau^\frac12\kappa_j+1&0\\ 0&\tau^\frac12\varepsilon^2\kappa_j+1
   \end{bmatrix}^{-1}
     \begin{bmatrix}
     \tau^\frac12\kappa_j & 1\\ 1& -3\tau^\frac12-\tau^\frac12\varepsilon^2\kappa_j
     \end{bmatrix}\\
   &=\begin{bmatrix}
    \d  \frac{\tau^\frac12\kappa_j }{\tau^\frac12\kappa_j+1}&\d
    \frac{1}{\tau^\frac12\kappa_j+1}\\[20pt]
   \d  \frac{1}{\tau^\frac12\varepsilon^2\kappa_j+1}& \d
 \frac{-3\tau^\frac12-\tau^\frac12\varepsilon^2\kappa_j }{\tau^\frac12\varepsilon^2\kappa_j+1}
      \end{bmatrix}
\end{align*}
 for $1\leq j\leq n_h$.  Hence we only need to understand the behavior
 of the eigenvalues of the matrix
\begin{equation*}
 \bC=\begin{bmatrix}
  \d  \frac{\omega}{\omega+1}&\d  \frac{1}{\omega+1}\\[20pt]
          \d  \frac{1}{\omega\varepsilon^2+1}& \d
          \frac{-3\tau^\frac12-\omega\varepsilon^2 }{\omega\varepsilon^2+1}
      \end{bmatrix},
\end{equation*}
 where $\omega$ is a positive number and $0<\tau,\varepsilon\leq 1$.
\par
 First of all we have
\begin{equation}\label{eq:P1}
  |\lambda|\leq \|\bC\|_\infty \leq 4
\end{equation}
 for any eigenvalue $\lambda$ of $\bC$, which implies that the second estimate in
 \eqref{eq:eigenvalue-condition} holds for $C_1=4$.
\par
  A direct calculation shows that
\begin{equation*}
 |\det\bC|=\frac{1+3\tau^\frac12\omega+\varepsilon^2\omega^2}
 {1+(1+\varepsilon^2)\omega+\varepsilon^2\omega^2}
 \geq \frac{1+3\tau^\frac12\omega+\varepsilon^2\omega^2}
 {1+2\omega+\varepsilon^2\omega^2}.
\end{equation*}
 On one hand we have
\begin{equation*}
 1+2\omega+\varepsilon^2\omega^2
   \leq \tau^{-\frac12}(1+3\tau^\frac12\omega+\varepsilon^2\omega^2),
\end{equation*}
 which implies
\begin{equation}\label{eq:P2}
  |\det\bC|\geq \tau^\frac12.
\end{equation}
 On the other hand we also have
\begin{equation*}
  1+2\omega+\varepsilon^2\omega^2\leq
  \varepsilon^{-1}(1+2\varepsilon\omega+\varepsilon^2\omega^2)
  \leq 2\varepsilon^{-1}(1+\varepsilon^2\omega^2)
   \leq 2\varepsilon^{-1}(1+3\tau^\frac12\omega+\varepsilon^2\omega^2),
\end{equation*}
 which implies
\begin{equation}\label{eq:P3}
|\det\bC|\geq \frac{\varepsilon}{2}.
\end{equation}
\par
 Putting \eqref{eq:P1}--\eqref{eq:P3} together we see that
\begin{equation*}
  4|\lambda|\geq|\det \bC|\geq \max(\tau^\frac12,\varepsilon/2)
\end{equation*}
 for any eigenvalue $\lambda$ of $\bC$.  Therefore the first estimate in
 \eqref{eq:eigenvalue-condition} holds with $C_2=1/8$.
\end{proof}
\par
 In our numerical experiments we use the preconditioner $\bP_*$ given by
\begin{equation}\label{eq:PStar}
 \bP_*= \begin{bmatrix}
\tau^\frac{1}{2}\bK + \bM & {\bf 0} \\
{\bf 0} & \tau^\frac{1}{2} \varepsilon^2 \bK+ \bM
\end{bmatrix}.
\end{equation}
 Since the two symmetric positive definite
 matrices $\bP$ and $\bP_*$ are spectrally equivalent, we immediately deduce from
 Theorem~\ref{thm:Preconditioner} that there exist two positive constants $C_3$ and $C_4$ independent of
 $\varepsilon$, $h$ and $\tau$ such that
\begin{equation}\label{eq:PrecondEst}
   C_4\max(\tau^\frac12,\varepsilon)\leq |\lambda|\leq C_3
\end{equation}
 for any eigenvalue $\lambda$ of $\bP_*^{-1}\bB$.
\par
 According to \eqref{eq:PrecondEst}, the performance of the  preconditioned
 MINRES algorithm (cf. \cite{Greenbaum:1997:Iterative,ESW:2014:FEFIS})
 for systems involving $\bf B$ is independent of $\tau$ and $h$ for a given $\varepsilon$,
 and also independent of $\varepsilon$ and $h$ for a given $\tau$.
  Similar behavior can also be expected for systems involving the
 matrix in \eqref{eq:SystemMatrix}.  Furthermore, the action of $(\gamma \bK + \bM)^{-1}$
  on a vector can be computed by a multigrid method, which creates large computational savings.
\begin{rem}\label{rem:varepsilon}
 Recall the matrix $\bB$ is obtained from the matrix in \eqref{eq:SystemMatrix} by replacing
 ${\bf J}(\phi_{h,j}^m)$ by $3\bM$ and its justification depends on $\varepsilon$.
 Therefore we expect to see some dependence of
 the performance of the preconditioned MINRES algorithm on $\varepsilon$
 for a given $\tau$.
\end{rem}
\begin{rem}\label{rem:SmallTau}
   When $\tau$ becomes $0$, the matrix
     $$\bB=\begin{bmatrix}
       0 & \bM \\
       \bM & 0
     \end{bmatrix}$$
 is well-conditioned. Therefore the performance of the preconditioned MINRES algorithm
 for systems involving the matrix in \eqref{eq:SystemMatrix}
 will improve as the time step size decreases.
\end{rem}
\begin{rem}\label{rem:SPP}
 Block diagonal preconditioners for saddle point systems are discussed in
 {\rm\cite{BGL:2005:SaddlePoint,MW:2011:SaddlePoint}} and the references therein.
\end{rem}
\section{Numerical Experiments}
\label{sec:Numerics}
 In this section we report the reulsts of
 eight numerical experiments in two and three dimensions.
   All computations were carried out using the FELICITY MATLAB/C++ Toolbox \cite{felicity}
  unless specified otherwise.
\par
 In the first six numerical experiments, we solve \eqref{subeqs:Schemea} on the
 unit square $\Omega=(0,1)^2$ using uniform meshes.
 The initial mesh $\mathcal{T}_0$ is generated by the two diagonals of
 $\Omega$  and the meshes $\mathcal{T}_1,\mathcal{T}_2,\ldots$ are obtained
 from $\mathcal{T}_0$ by uniform refinements.
\par
 For the first five experiments, we use the initial data
\begin{equation}
\phi_{h}^0 = \mathcal{I}_h
\Big[\Big(\frac{1}{2}\Big)[1 - \cos(2\pi x_1)][1 - \cos(2\pi x_2))]- 1\Big] ,
\end{equation}
 where $\mathcal{I}_h :   H^2(\Omega) \longrightarrow S_h$ is the
 standard nodal interpolation operator.
\par
 The system \eqref{subeqs:Schemea} (or equivalently \eqref{subeqs:Schemeb})
 is solved by the Newton iteration with a tolerance
 of $10^{-15}$ for $\|\delta_j\phi\|_{L^\infty(\Omega)}$ or
  a residual tolerance of $10^{-7}$
 for \eqref{subeqs:NewJacobian1}--\eqref{eq:NewJacobian2},
  whichever is satisfied first.  It turns out that only one Newton iteration is needed for
  each time step in all the experiments.
\par
  During each Newton iteration, the systems involving
  \eqref{eq:SystemMatrix} are solved by a preconditioned MINRES
  algorithm with a residual tolerance of $10^{-7}$.  The systems
  involving the preconditioner $\bf P$ are solved by
  a multigrid $V(4,4)$ algorithm that uses the Gauss-Seidel iteration as the smoother
  (cf. \cite{Hackbusch:1985:MMA,TOS:2001:MG}).
  In all our experiments the maximum number of preconditioned MINRES iterations occured
  during the first few time steps after which the number of iterations
  would decrease and level off.
\par
 In the first experiment, we take  $\tau = {0.002}/{64}$
 with a final  time $T=0.04$ for the two interfacial width parameters
 $\varepsilon = 0.0625$ and $\varepsilon=0.001$.
 In Table~\ref{ch-tab-fixed-tau-eps}  we report the average number of preconditioned MINRES
 iterations over all time steps
 along with the average solution
 time per time step as the mesh is refined.
 (The timing mechanism is the `tic toc' command in MATLAB.)
 We observe that the
 performance of the preconditioned MINRES algorithm does not depend on $h$ and the
 solution time per time step grows linearly with the number of degrees of freedom.
 Moreover the solution time roughly doubles as $\varepsilon$ decreases from $0.0625$
 to $0.001$, indicating that the performance of the solver only has a mild dependence
 on $\varepsilon$.
\begin{table}[H]
	\centering
	\begin{tabular}{c|cc|cc}
 & \multicolumn{2}{ c }{$\varepsilon = 0.0625$} &\multicolumn{2}{| c }{$\varepsilon = 0.001$}
 \\
$h$ & MINRES Its. & Time to Solve (s) & MINRES Its. & Time to Solve (s)
	\\
	\hline
$\nicefrac{\sqrt{2}}{8}$ & 20 & 0.042391 & 28  & 0.01786
	\\
$\nicefrac{\sqrt{2}}{16}$ & 21 & 0.070047 & 44  & 0.04537
	\\
$\nicefrac{\sqrt{2}}{32}$ & 23 & 0.156576 & 57  & 0.13569
	\\
$\nicefrac{\sqrt{2}}{64}$ & 24 & 0.444770 & 71  & 0.50508
	\\
$\nicefrac{\sqrt{2}}{128}$ & 25 & 1.752561 & 107 & 3.13307
\\
$\nicefrac{\sqrt{2}}{256}$ & 26 & 6.884936 & 96 & 12.3052
\\
$\nicefrac{\sqrt{2}}{512}$ & 26 & 26.84091 & 97  & 57.2141
\\
$\nicefrac{\sqrt{2}}{1024}$ & 26 & 108.9613 & 100 &  245.456
\\
\hline
\end{tabular}
\caption{The average number of preconditioned MINRES iterations over all
 time steps together with the average solution time per time step as the mesh
 is refined ($\Omega = (0,1)^2, \tau = \nicefrac{0.002}{64}$, $T =0.04$
 $\varepsilon=0.0625$ (left), $\varepsilon=0.001$ (right)).}
	\label{ch-tab-fixed-tau-eps}
\end{table}
\par
	Table~\ref{ch-tab-time-fenics} shows the
  average solution time per time step for the same problem with $\varepsilon=0.0625$
 using  FEniCS  \cite{fenics15} on the prebuilt high-performance Docker.
 The main components of this code are Newton's method and LU decomposition.
 The time step size is fixed at $\tau= {0.002}/{64}$.  The residual tolerance
 is set at $10^{-7}$. The timing mechanism is a start and stop of the python
 command `time.time'. By comparing Table \ref{ch-tab-time-fenics} with
  Table \ref{ch-tab-fixed-tau-eps}, we see that FEniCS  appears to be
  faster for the coarser mesh sizes.
   However, as the mesh is refined, the advantage of our method is clearly observed.
\begin{table}[h!]
\centering
\begin{tabular}{ccccccc}
$h$ & Avg. Time to Solve (s)
	\\
	\hline
$\nicefrac{\sqrt{2}}{8}$  & 0.01449
	\\
$\nicefrac{\sqrt{2}}{16}$  & 0.02589
	\\
$\nicefrac{\sqrt{2}}{32}$  & 0.06675
	\\
$\nicefrac{\sqrt{2}}{64}$ & 0.21948
	\\
$\nicefrac{\sqrt{2}}{128}$  & 3.94694
\\
$\nicefrac{\sqrt{2}}{256}$  & 44.4569
\\
	\hline
	\end{tabular}
	\caption{The average solution time per time step using FEniCS to run the same test as performed in Table \ref{ch-tab-fixed-tau-eps} with $\varepsilon = 0.0625$}
	\label{ch-tab-time-fenics}
	\end{table}

\par	
 In the second experiment, we again take $\tau ={0.002}/{64}$ and a final time $T=0.04$.
 The median numbers of the preconditioned MINRES iterations over all time steps
 for several values of $\varepsilon$ as
 the mesh is refined are plotted in Figure~\ref{fig:h_and_eps_Med}.
 The performance of our method is independent of the mesh size $h$, and there is some dependence on
 the interfacial width parameter $\varepsilon$ as expected (cf. Remark~\ref{rem:varepsilon}).
\begin{figure}[H]
\hspace{-.3in}
\begin{minipage}{7in}
{
\includegraphics[width = 7in]{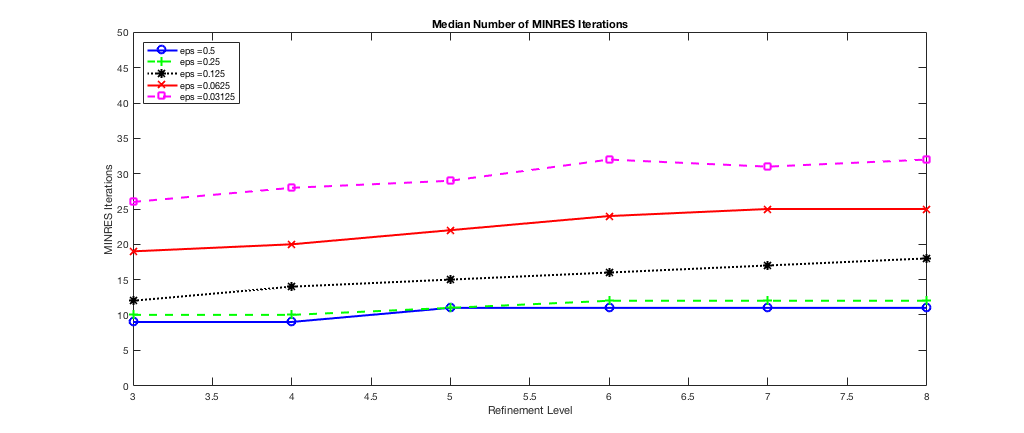}}
\end{minipage}
\caption{The median number of MINRES iterations over all time steps
for several values of
 $\varepsilon$ as the mesh is refined ($\Omega=(0,1)^2$,
 $\tau=0.002/64$ and $T=0.04$).}
\label{fig:h_and_eps_Med}
\end{figure}
\par
 In the third experiment, we fix $h=\sqrt2/64$, a final time
 $T=.04$, $\varepsilon=0.0625$ and $0.001$,  and refine the time step size $\tau$.
 The average number of the preconditioned MINRES iterations over all
 time steps is displayed in
 Table~\ref{ch-tab-fixed-h-tau} along with the average solution time per time step.
 The performance is clearly independent of the time step size $\tau$.
 The solution time roughly triples as $\varepsilon$ decreases from $0.0625$ to
 $0.001$, indicating again that the performance of the solver only depends mildly
 on $\varepsilon$.
\begin{table}[H]
	\centering
	\begin{tabular}{c|cc|cc}
& \multicolumn{2}{ c }{$\varepsilon = 0.0625$}  &  \multicolumn{2}{| c }{$\varepsilon = 0.001$}
 \\
$\tau$ & MINRES Its. & Time to Solve (s) & MINRES Its. & Time to Solve (s)
\\
\hline
$\nicefrac{.002}{8}$ & 27 & 0.310728 & 55 & 0.5297575
	\\
$\nicefrac{.002}{16}$ & 25 & 0.282849 & 57 & 0.5524107
	\\
$\nicefrac{.002}{32}$ & 25 & 0.273073 & 67&  0.6346132
	\\
$\nicefrac{.002}{64}$ & 24 & 0.265608 & 71 & 0.6825223
	\\
$\nicefrac{.002}{128}$ & 23 & 0.260789 & 71 & 0.6788814
	\\
$\nicefrac{.002}{256}$ & 22 & 0.251007 & 78 & 0.7282336
	\\
$\nicefrac{.002}{512}$ & 20 & 0.237420 & 70 & 0.6503484
\\
$\nicefrac{.002}{1024}$ & 18 & 0.219060 & 77 & 0.7032321
\\
	\hline
	\end{tabular}
	\caption{The average number of preconditioned MINRES iterations over
all time steps together with the average solution time per time step as the time step is refined
 ($\Omega = (0,1)^2, h = \nicefrac{\sqrt{2}}{64}$, $T = 0.04$, $\varepsilon = 0.0625$
  (left), $\varepsilon = 0.001$ (right)).}
	\label{ch-tab-fixed-h-tau}
	\end{table}

\par
 In the fourth experiment, we fix $\varepsilon=0.0625$ and a final time $T=0.04$.
  The median numbers of preconditioned MINRES iterations over all time steps
 for several values of $\tau$ as
 the mesh is refined are displayed in Figure~\ref{fig:h_and_tau_Med}.
 The performance of our method is clearly independent of the mesh size $h$ and the time step size $\tau$.
\begin{figure}[H]
\hspace{.5in}
\begin{minipage}{5.5in}
{
\includegraphics[width = 5.2in]{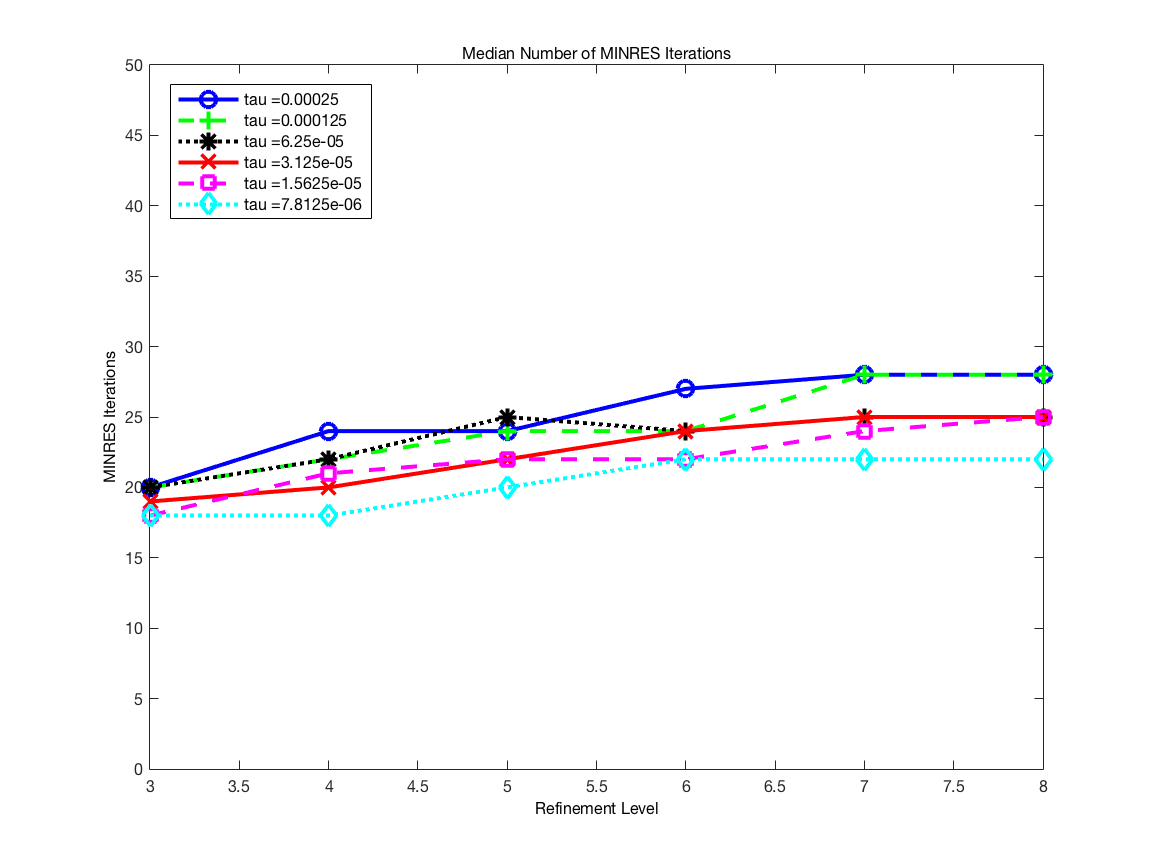}}
\end{minipage}
\caption{The median number of preconditioned MINRES iterations for several
  time step sizes as the mesh is refined
 ($\Omega=(0,1)^2$, $\varepsilon=0.0625$ and $T=0.04$). }
\label{fig:h_and_tau_Med}
\end{figure}	
\par
 In the fifth experiment, we fix the final time $T=0.04$ and let $\tau=0.002 h/\sqrt2$
 (cf. \eqref{eq:ErrorEstimate}).
   The median numbers of preconditioned MINRES iterations over all time steps
 for several values of $\varepsilon$ as
 the mesh is refined are displayed in
  Figure~\ref{fig:h_and_eps_and_tau_Med}.
  Again, the performance only depends on the interfacial width parameter $\varepsilon$.
\begin{figure}[H]
\hspace{1in}
\begin{minipage}{5in}
{\includegraphics[width = 4.3in]{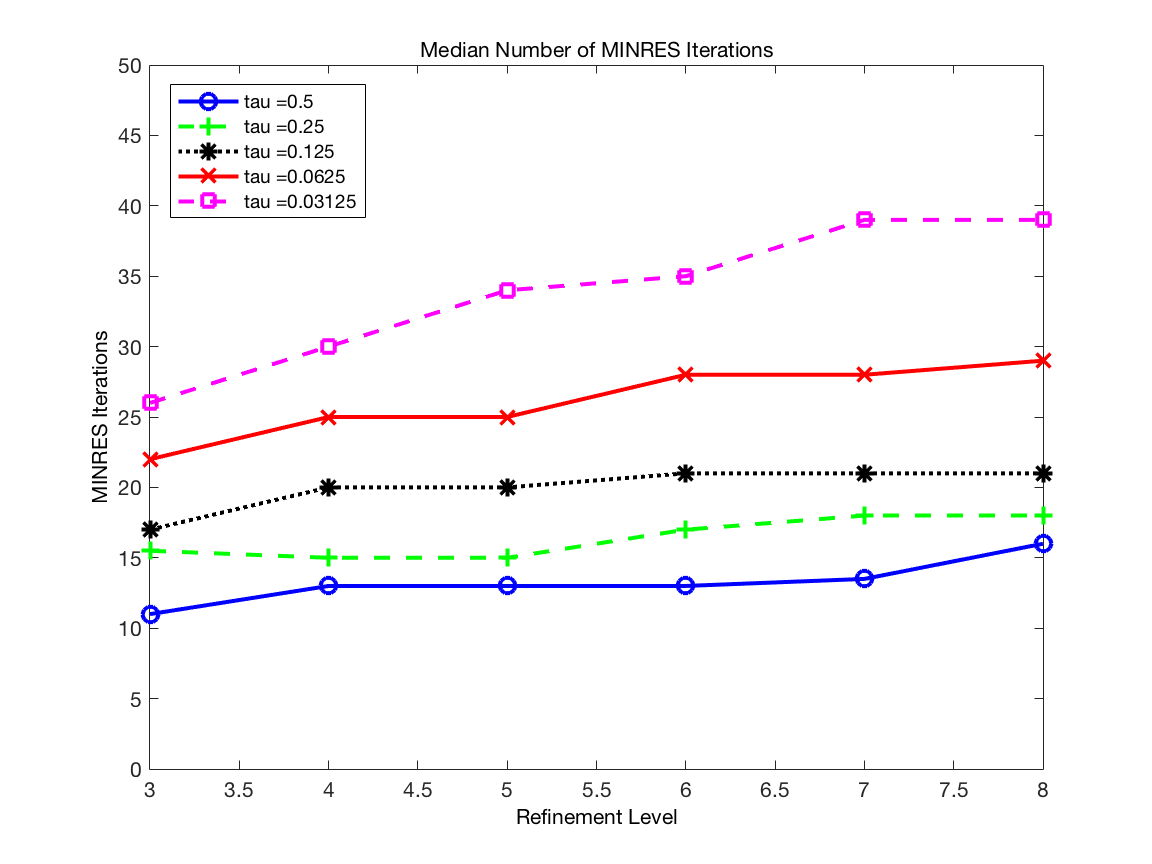}}
\end{minipage}
\caption{The median number of preconditioned MINRES iterations
 for several values of $\varepsilon$ as the mesh and time step are refined
 ($\Omega=(0,1)^2$, $\tau=0.002 h/\sqrt2$ and $T=0.04$).}
\label{fig:h_and_eps_and_tau_Med}
\end{figure}
\par
 In the sixth experiment, we solve the Cahn-Hillard equation with a random initial condition.
 We take $h=\sqrt2/128$, $\tau=0.002/128$ and $\varepsilon=0.0625$.  The surface plots for $\phi$
 at $t=0$, $t=0.0025$, $t=0.005$, $t=0.0075$, $t=0.01$ and $t=0.0125$ are displayed in
 Figure~\ref{fig:rand-matlab}.
 For comparison we solve the same problem using FEniCS and display the corresponding surface plots
 in Figure~\ref{fig:rand-fenics}.   The two figures are essentially indistinguishable.
\begin{figure}[H]
\subfloat{\includegraphics[width = 2in]{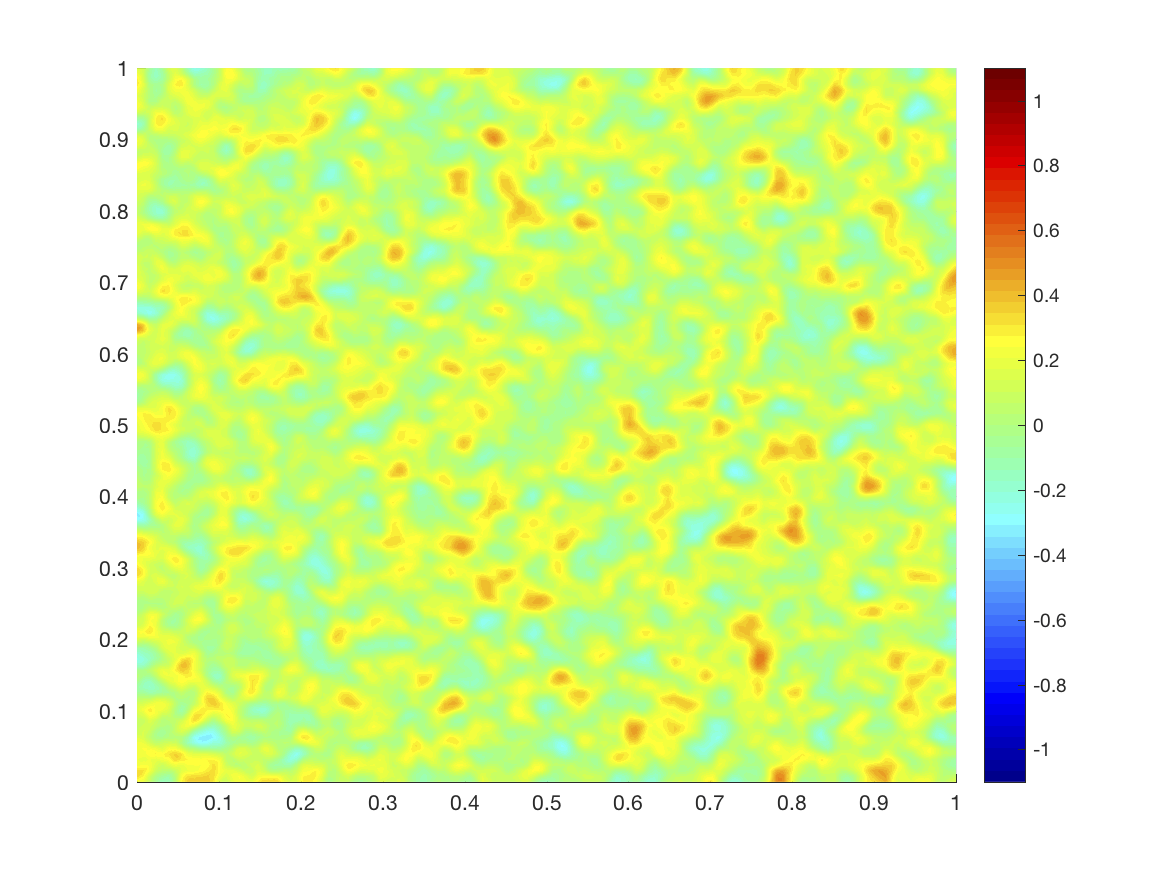}}
\subfloat{\includegraphics[width = 2in]{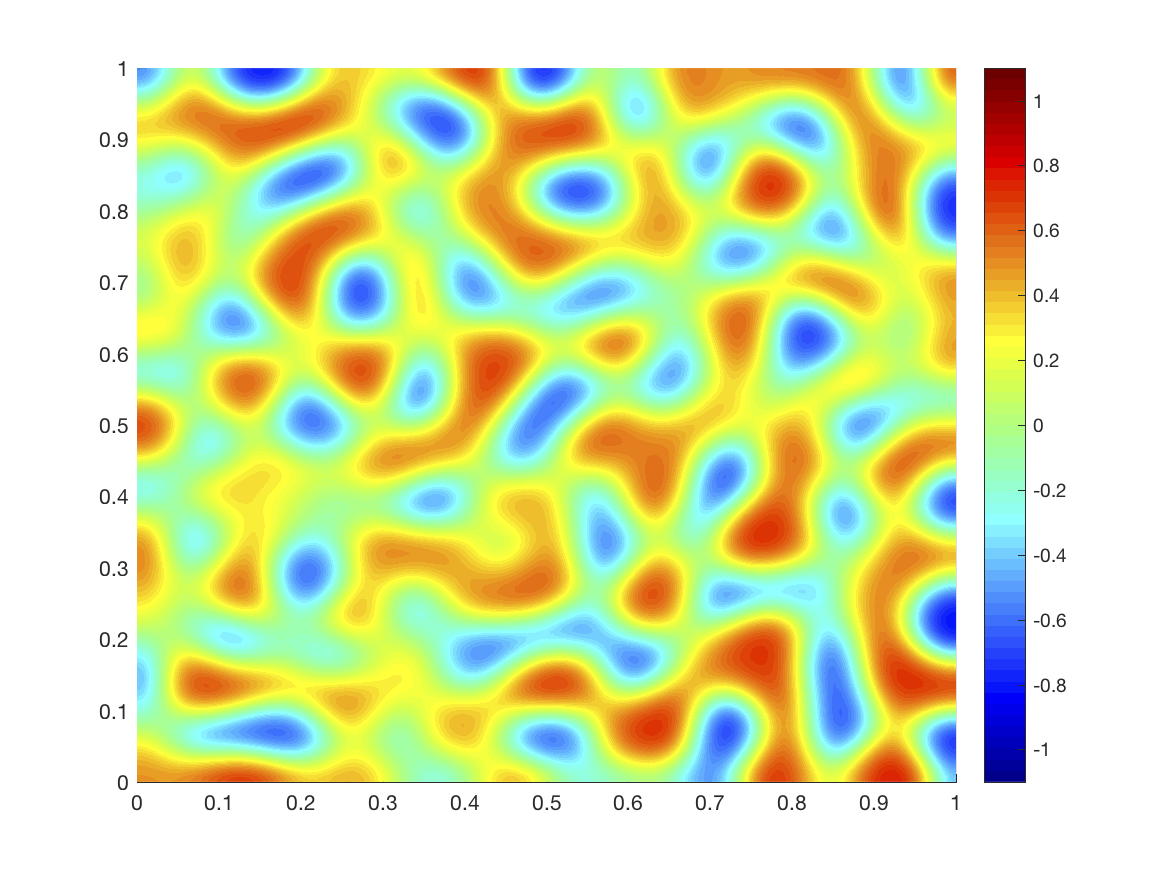}}
\subfloat{\includegraphics[width = 2in]{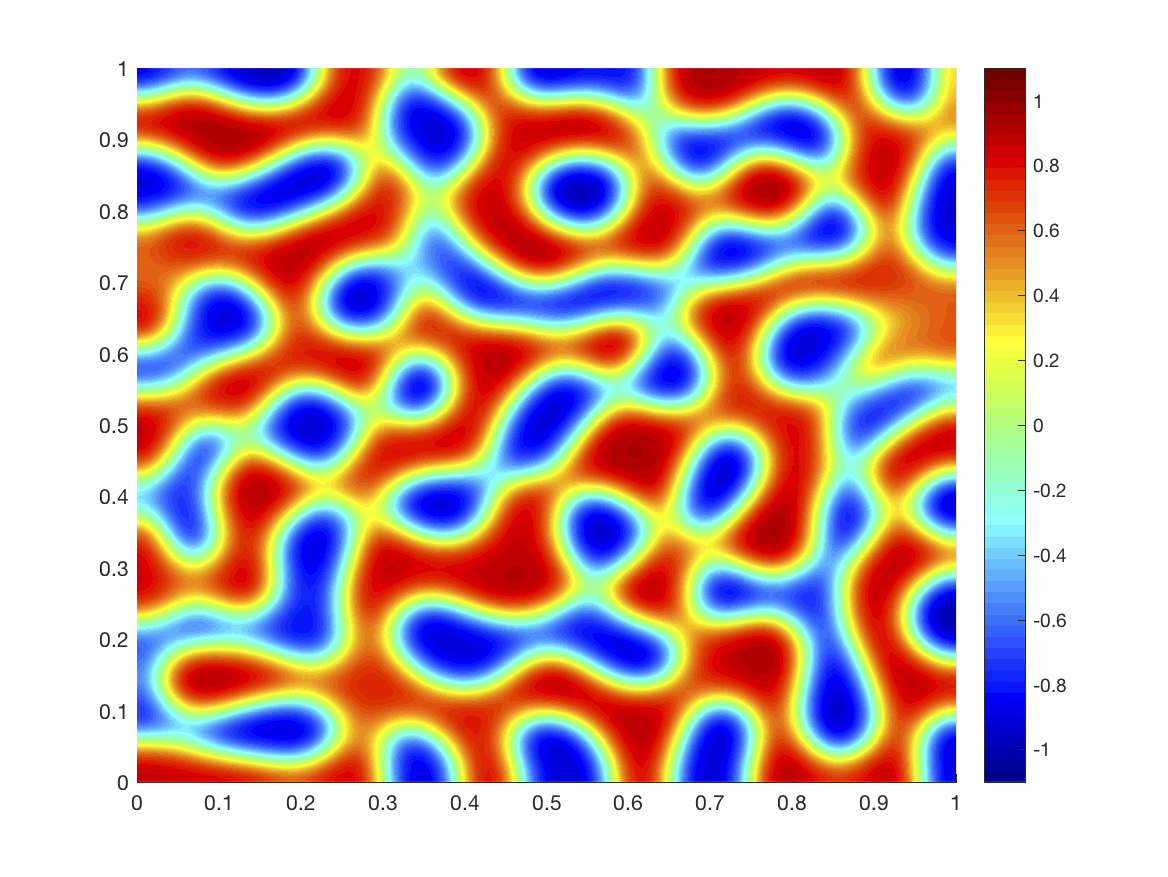}} \\
\subfloat{\includegraphics[width = 2in]{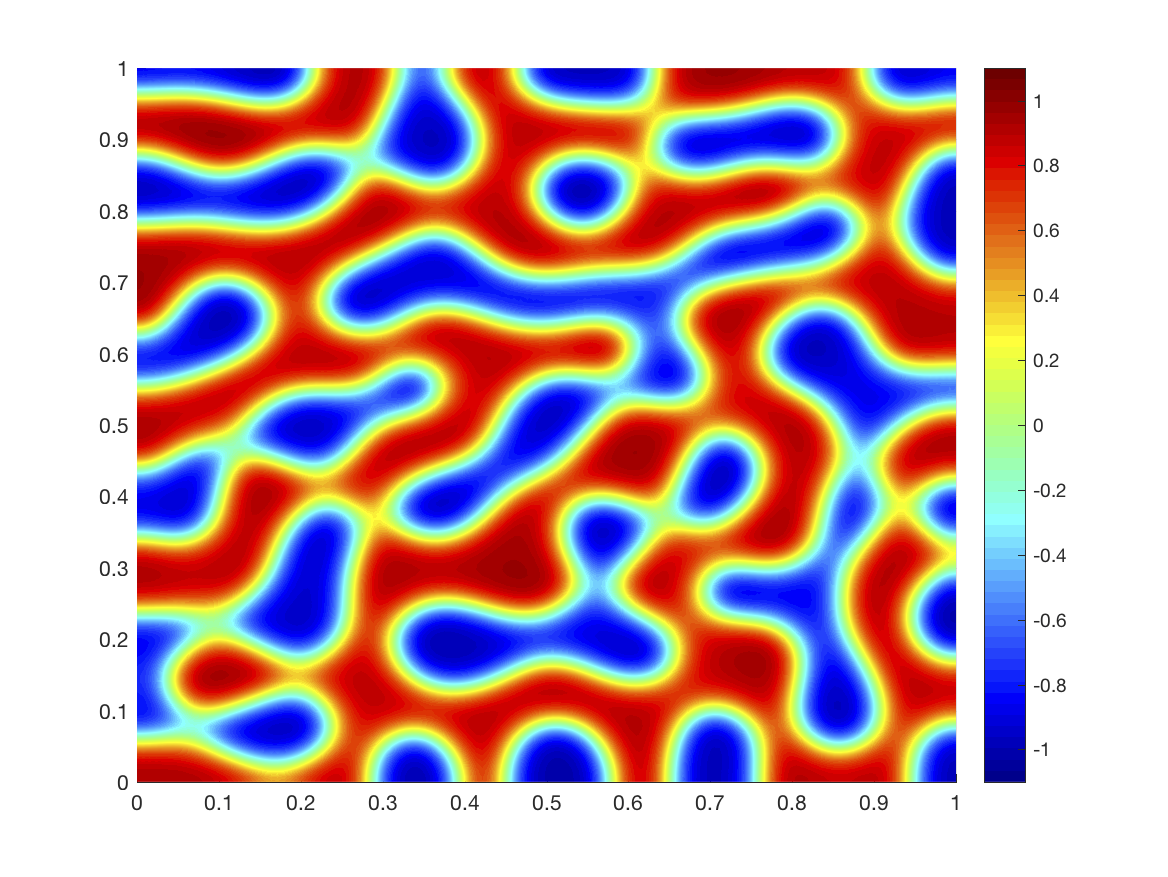}}
\subfloat{\includegraphics[width = 2in]{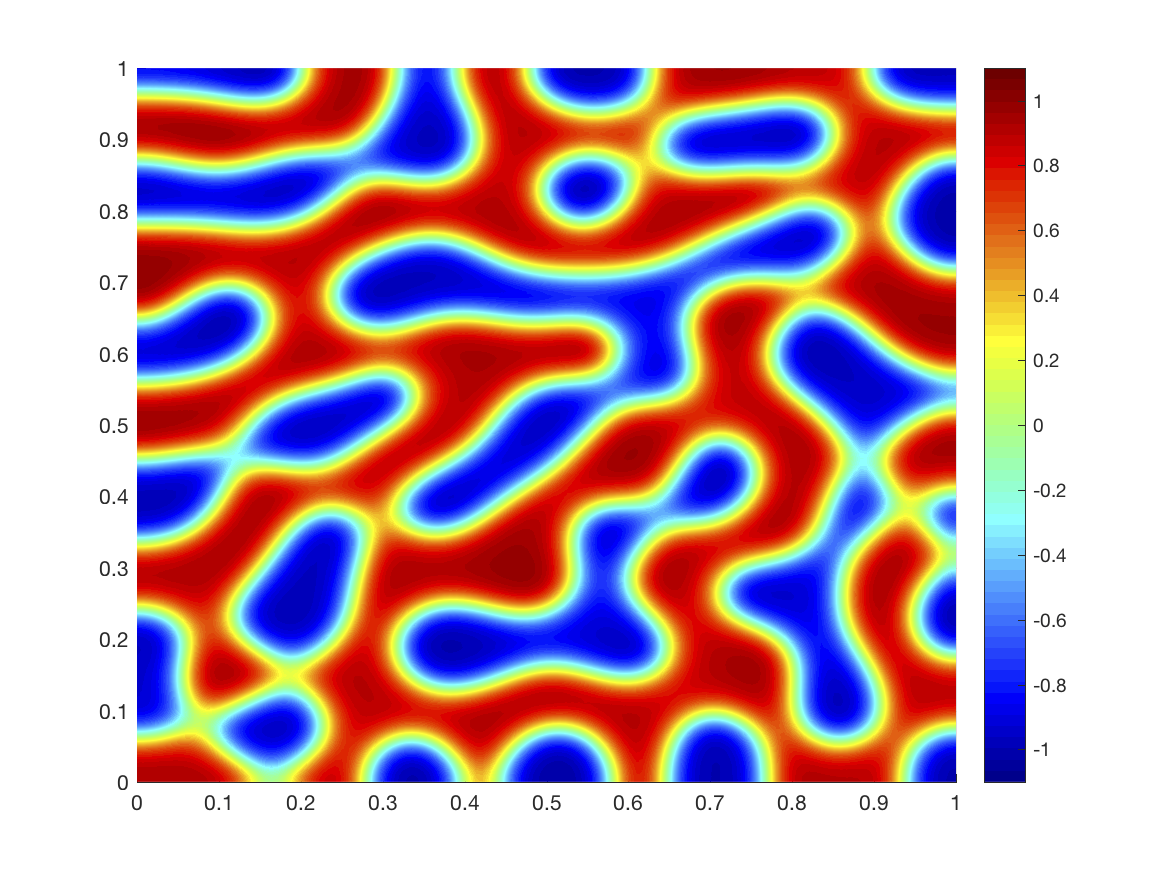}}
\subfloat{\includegraphics[width = 2in]{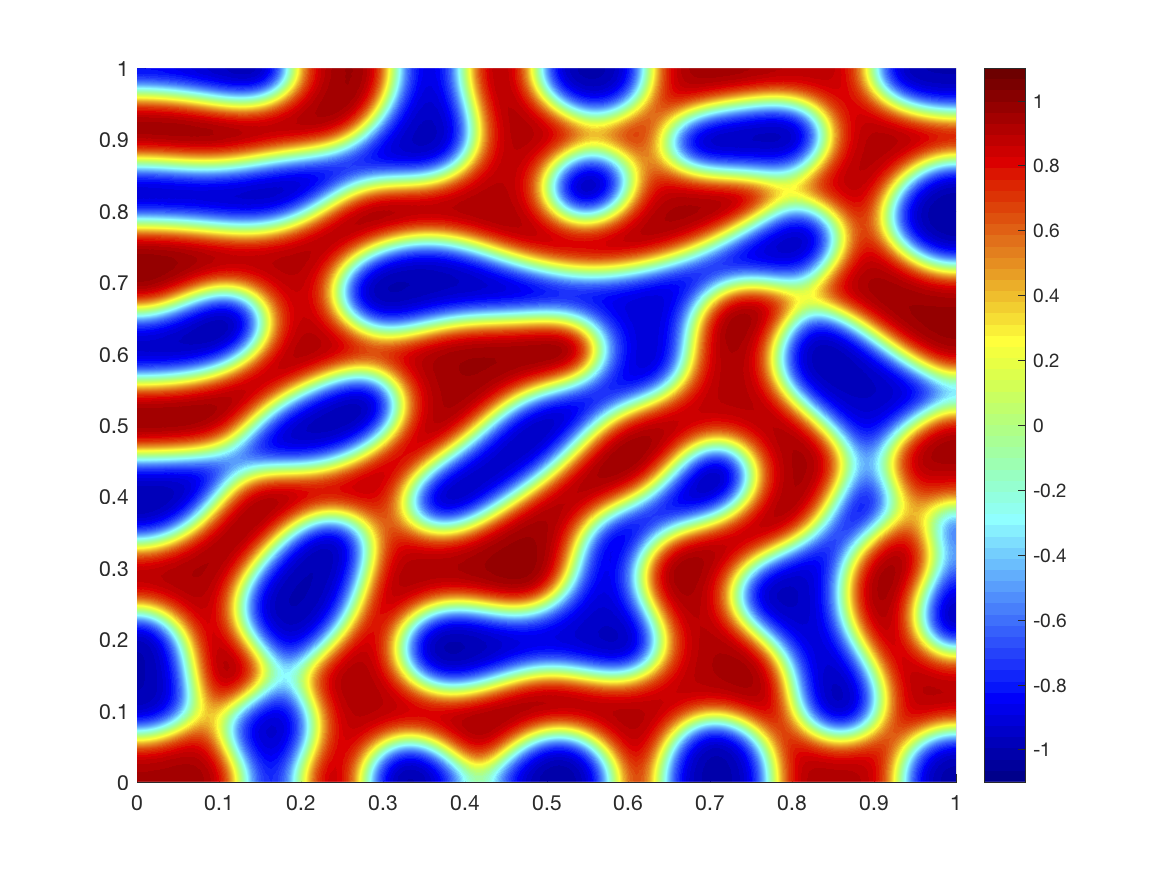}}\\
\caption{Spinodal decomposition of a binary fluid on $(0,1)^2$ with random initial data.
 The times displayed are $t=0, t=0.0025, t=0.005$ (top from left to right) and $t=0.0075, t=0.01, t=0.0125$ (bottom from left to right).}
\label{fig:rand-matlab}
\end{figure}
\begin{figure}[H]
\subfloat{\includegraphics[width = 2in]{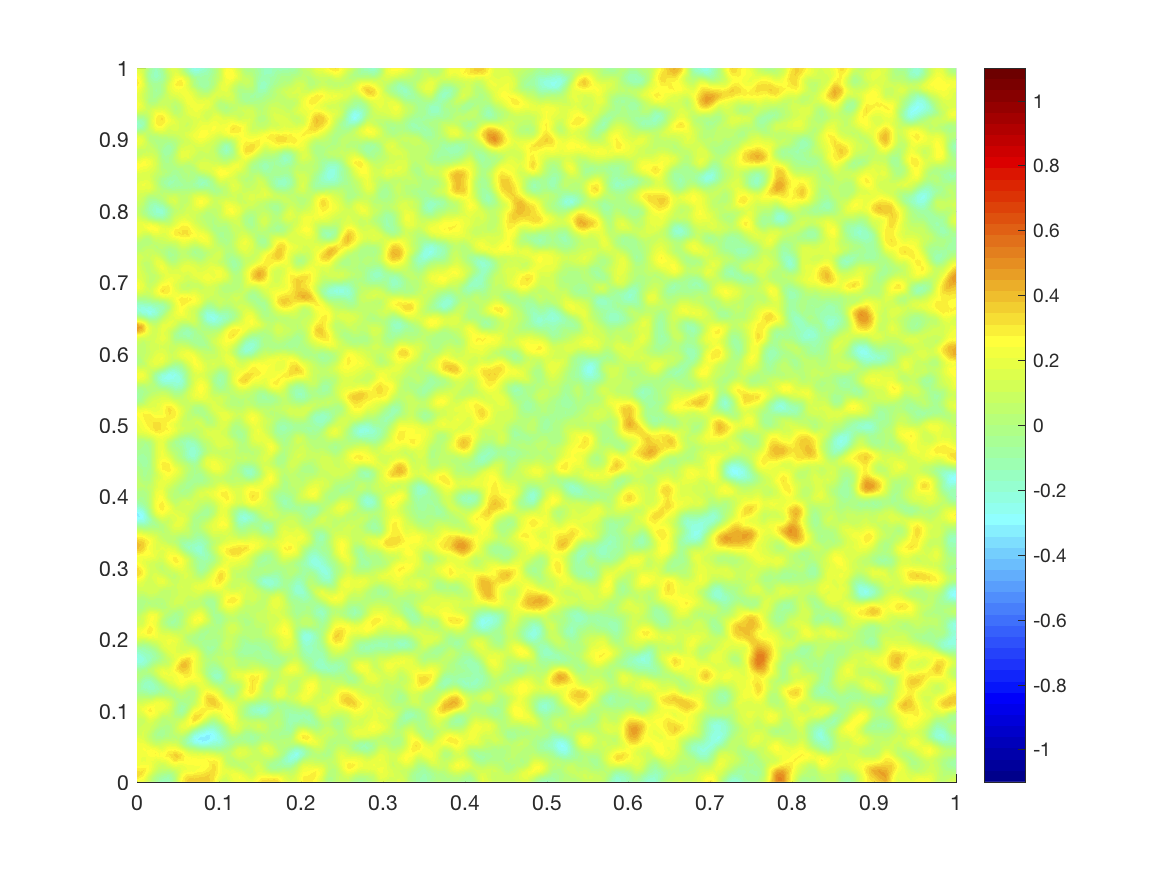}}
\subfloat{\includegraphics[width = 2in]{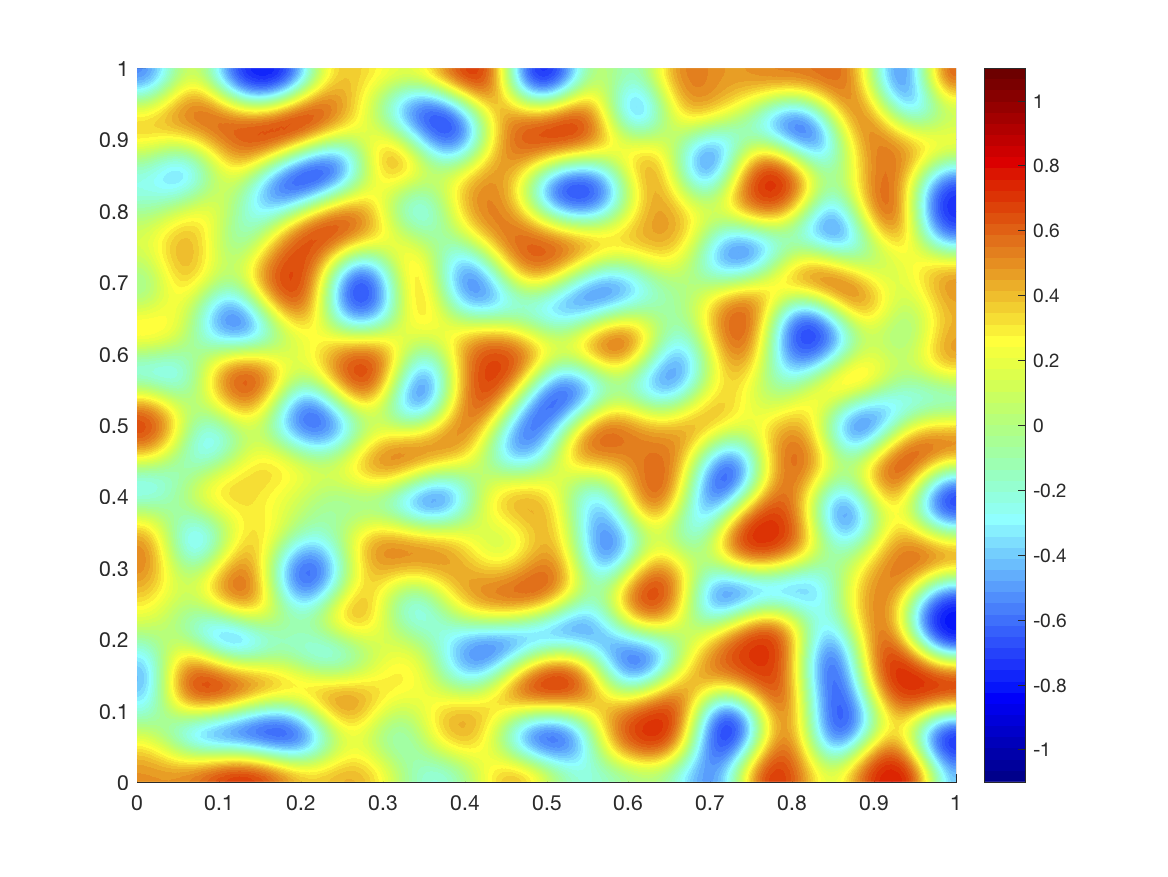}}
\subfloat{\includegraphics[width = 2in]{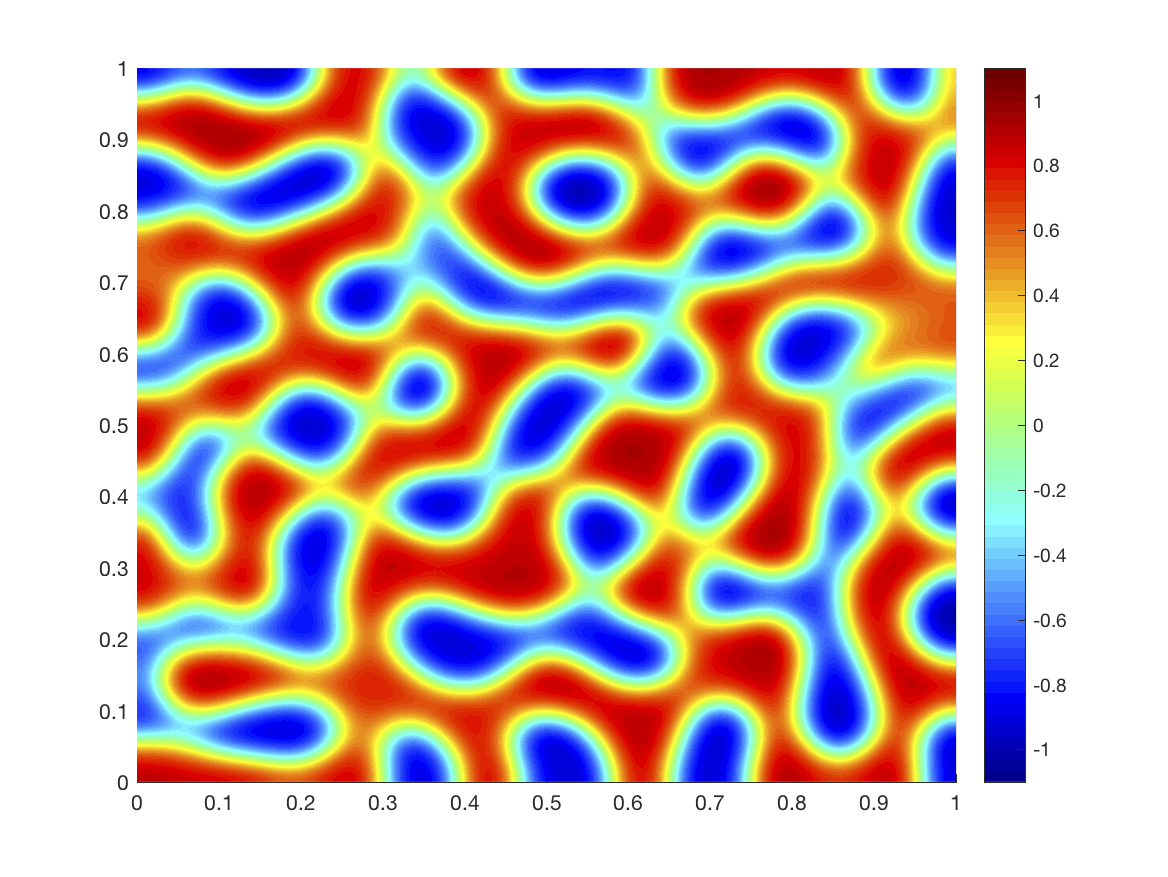}} \\
\subfloat{\includegraphics[width = 2in]{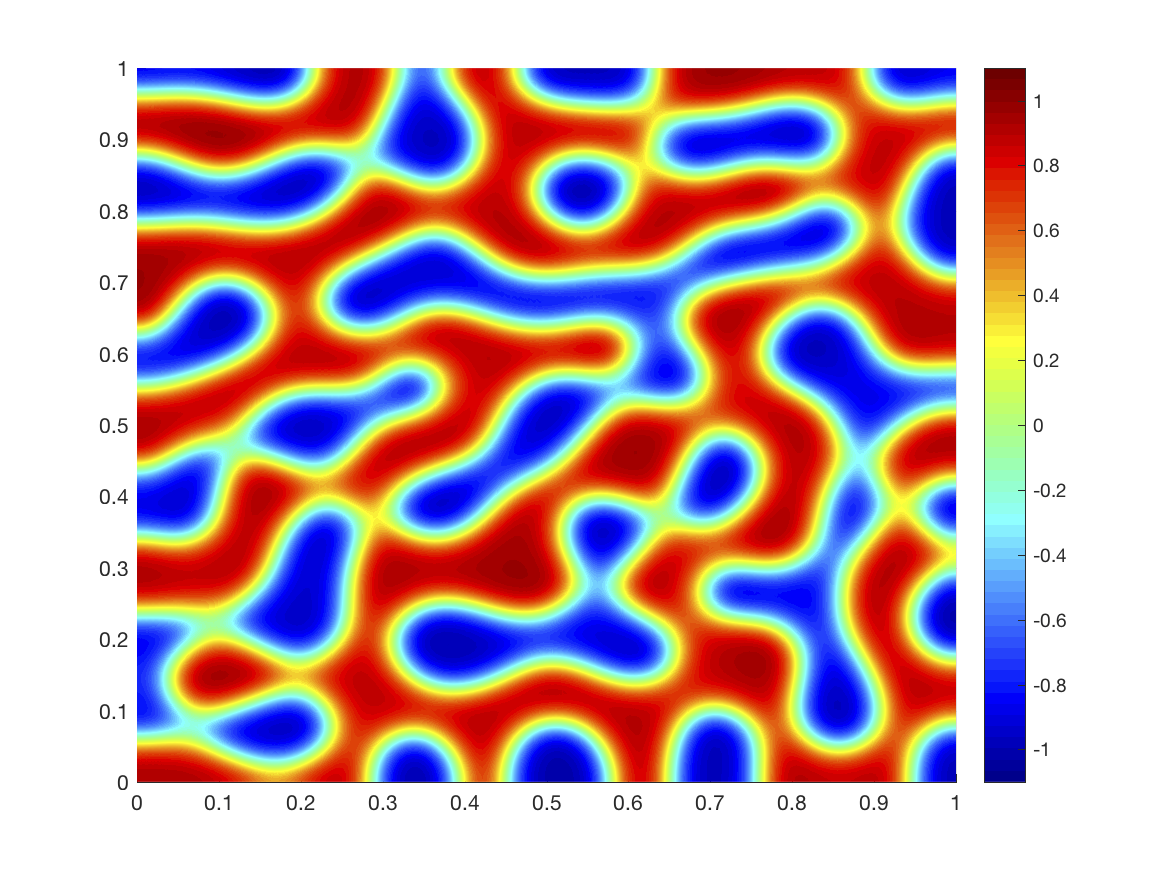}}
\subfloat{\includegraphics[width = 2in]{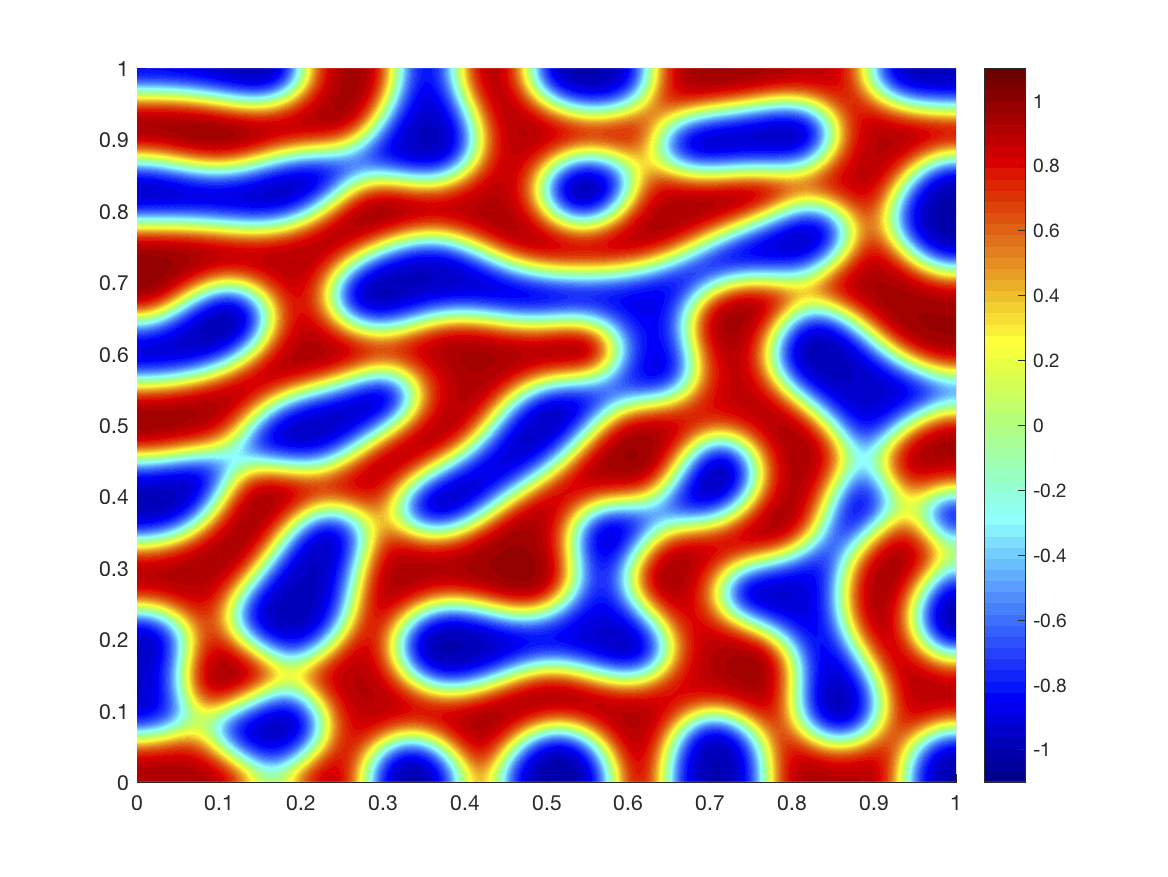}}
\subfloat{\includegraphics[width = 2in]{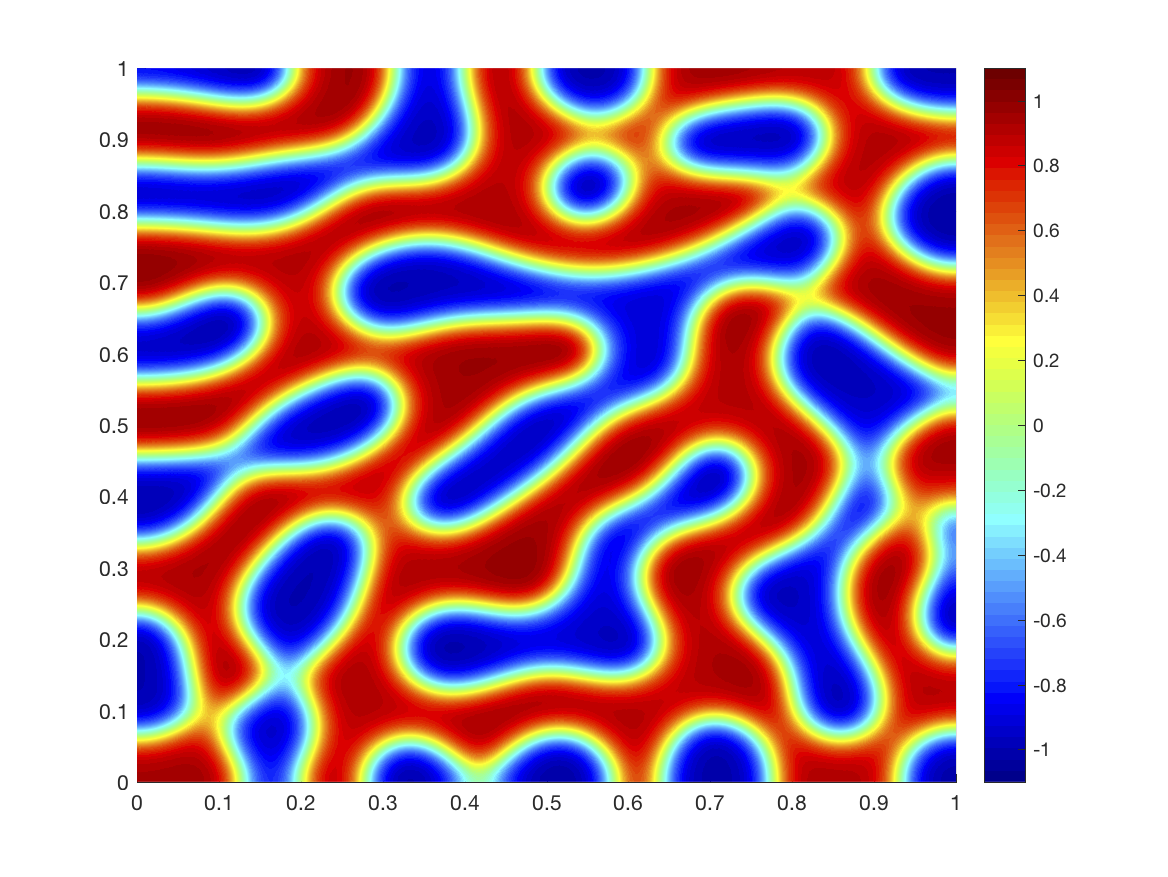}}\\
\caption{Spinodal decomposition of a binary fluid on $(0,1)^2$ with random initial data obtained by FEniCS.
 The times displayed are $t=0, t=0.0025, t=0.005$ (top from left to right) and $t=0.0075, t=0.01, t=0.0125$ (bottom from left to right).}
\label{fig:rand-fenics}
\end{figure}
\par
 In the seventh experiment, we solve the Cahn-Hilliard equation with a random
 initial condition on the unit cube $\Omega = (0,1)^3$ using uniform meshes.
 The initial mesh $\mathcal{T}_0$
 consists of six tetrahedrons. The meshes $\mathcal{T}_1, \mathcal{T}_2, \cdots$
 are obtained from $\mathcal{T}_0$ by uniform refinements.
 We take $\varepsilon=0.0625$, $\tau=0.002/64$, a final time $T=0.03$ and
 refine the mesh.
\par
 Table~\ref{ch-tab-3D} displays the maximum, median, and average number of
  preconditioned
 MINRES iterations
 over all time steps along with
 the average solution time per time step.
 Again, only one Newton iteration is needed for each time step.
 We observe that the
 performance of the preconditioned MINRES algorithm does not depend on $h$ and the
 solution time per time step grows linearly with the number of degrees of freedom.
\begin{table}[H]
	\centering
	\begin{tabular}{cccccc}
& \multicolumn{3}{ c }{MINRES Iterations}  & {Time to Solve (s)}
 \\
$h$ & Max. & Med. & Avg. & Avg.
\\
\hline
$\nicefrac{\sqrt{3}}{8}$ & 33 & 28 & 27  & .0460738
	\\
$\nicefrac{\sqrt{3}}{16}$ & 33 & 29 & 29  & .2451147
	\\
$\nicefrac{\sqrt{3}}{32}$ & 36 & 30 & 30  & 1.977465
	\\
$\nicefrac{\sqrt{3}}{64}$ & 37 & 29 & 30  & 15.97806
	\\
$\nicefrac{\sqrt{3}}{128}$ & 41 & 30 & 31  & 169.7146
	\\
	\hline
	\end{tabular}
	\caption{The maximum, median and average number of preconditioned MINRES iterations over
    all time steps together with the 
    average solution time per time step as the mesh is
    refined ($\Omega=(0,1)^3$, $\varepsilon=0.0625$, $\tau=0.002/64$ and $T=0.03$).}
	\label{ch-tab-3D}
	\end{table}
\par	
	Isocap plots for $\phi$ at $t=0, t=0.0015625, t=0.003125, t=0.0046875$,
  $t=0.00625$ and $t=0.0078125$ are displayed in Figure \ref{fig:rand-matlab-3D}.
	
\begin{figure}[H]
\subfloat{\includegraphics[width = 2in]{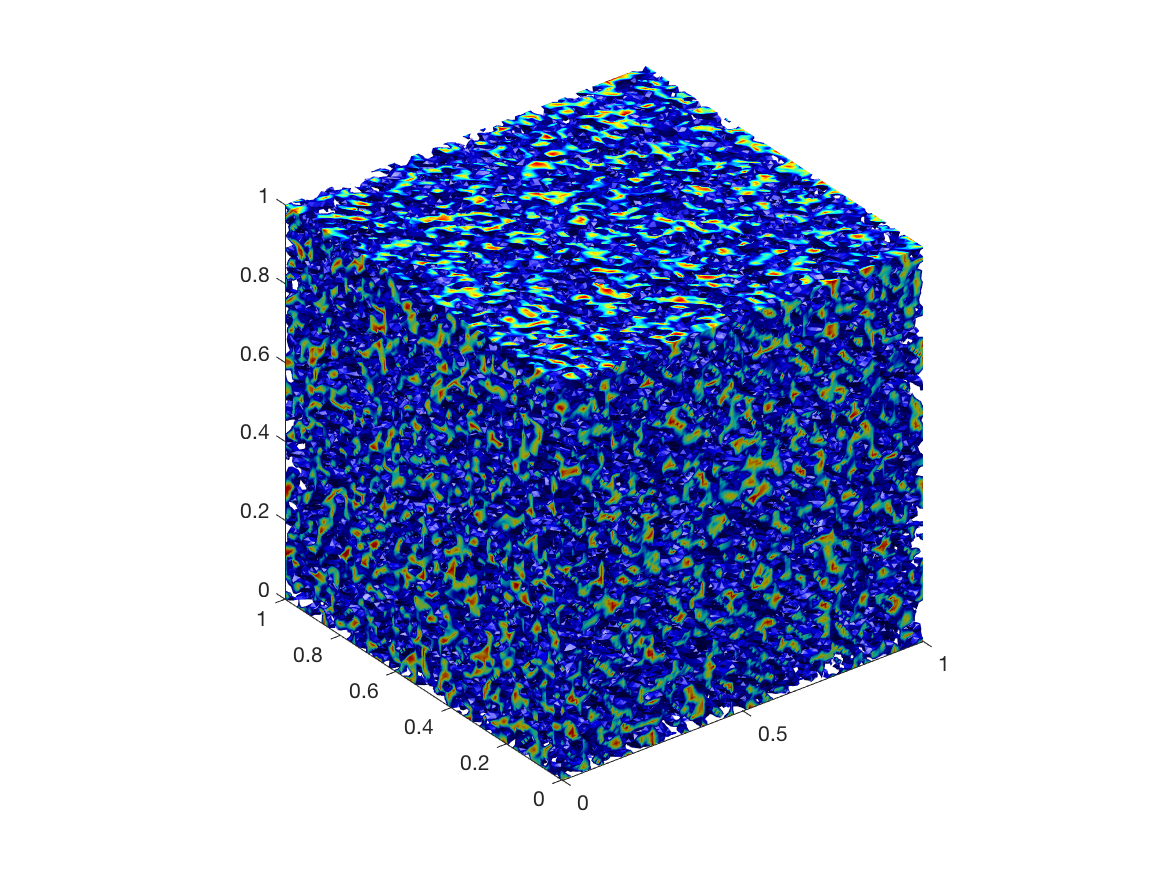}}
\subfloat{\includegraphics[width = 2in]{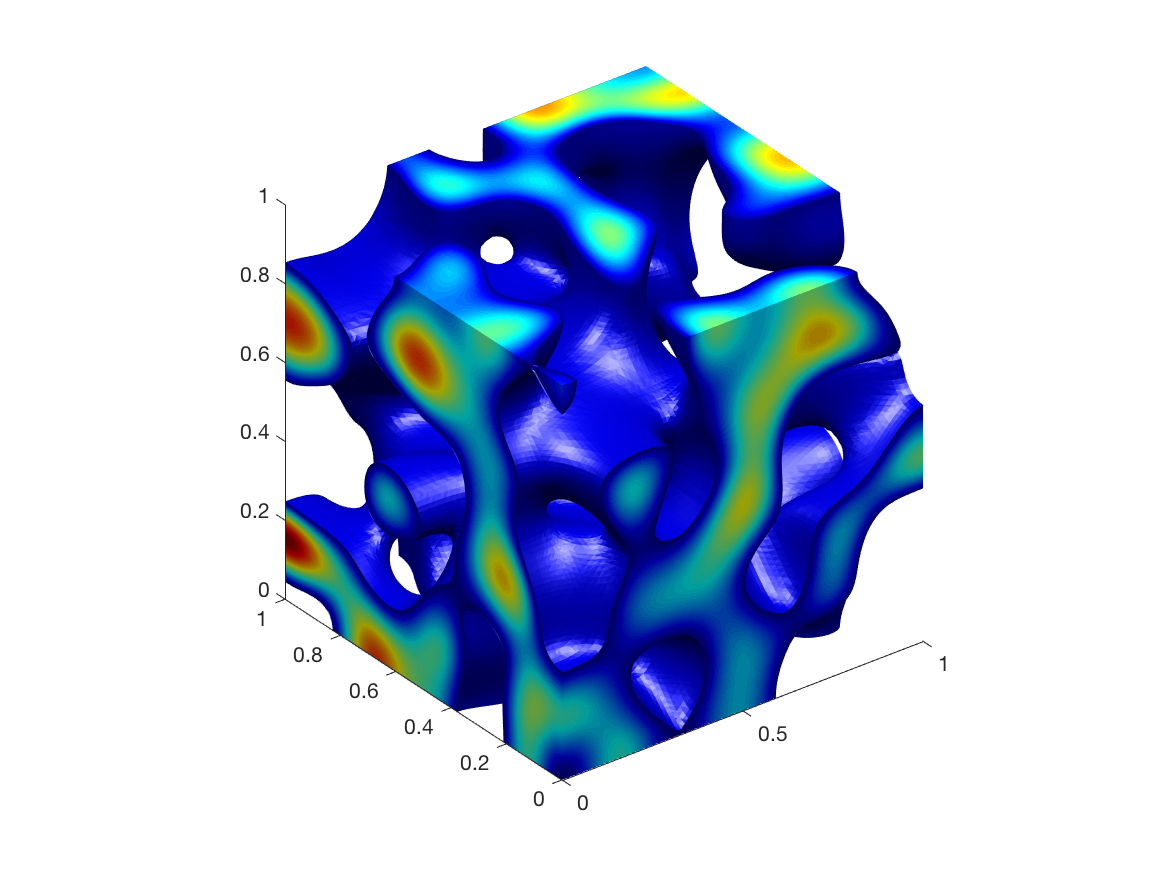}}
\subfloat{\includegraphics[width = 2in]{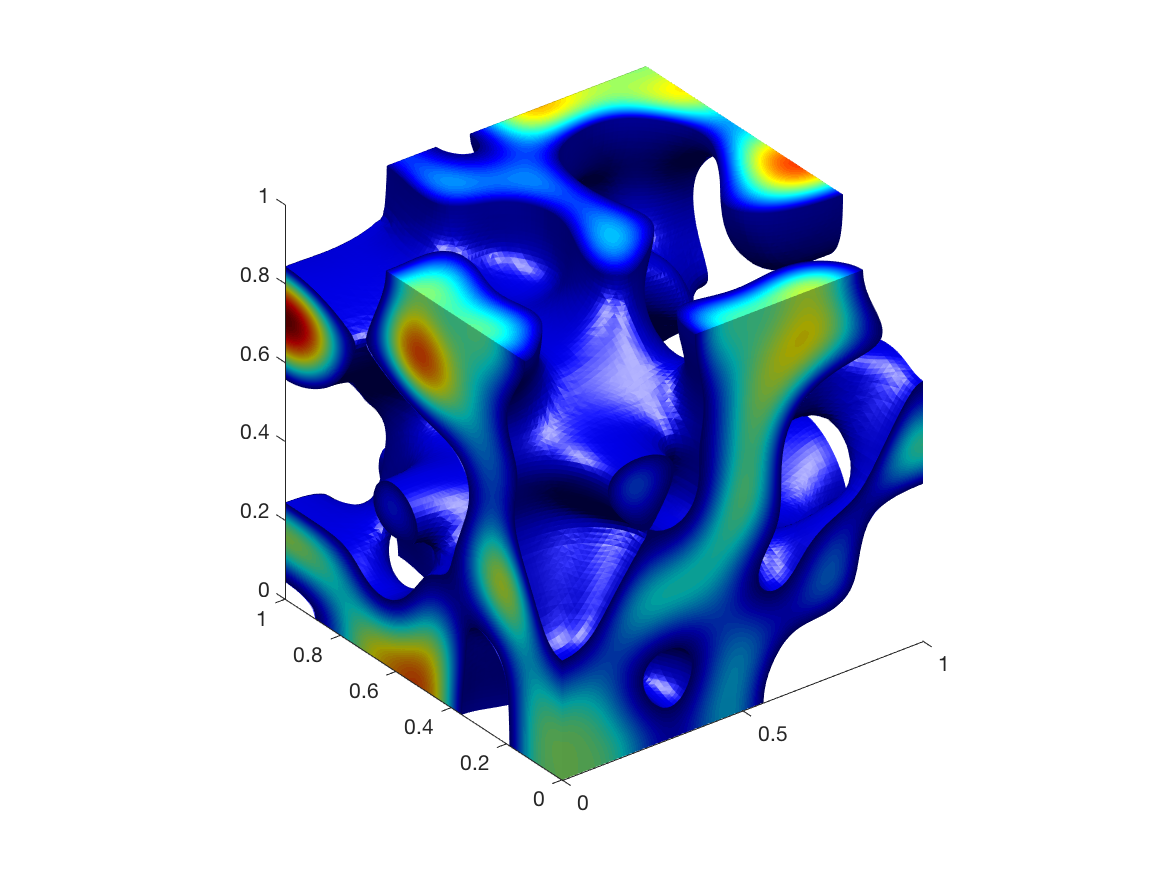}} \\
\subfloat{\includegraphics[width = 2in]{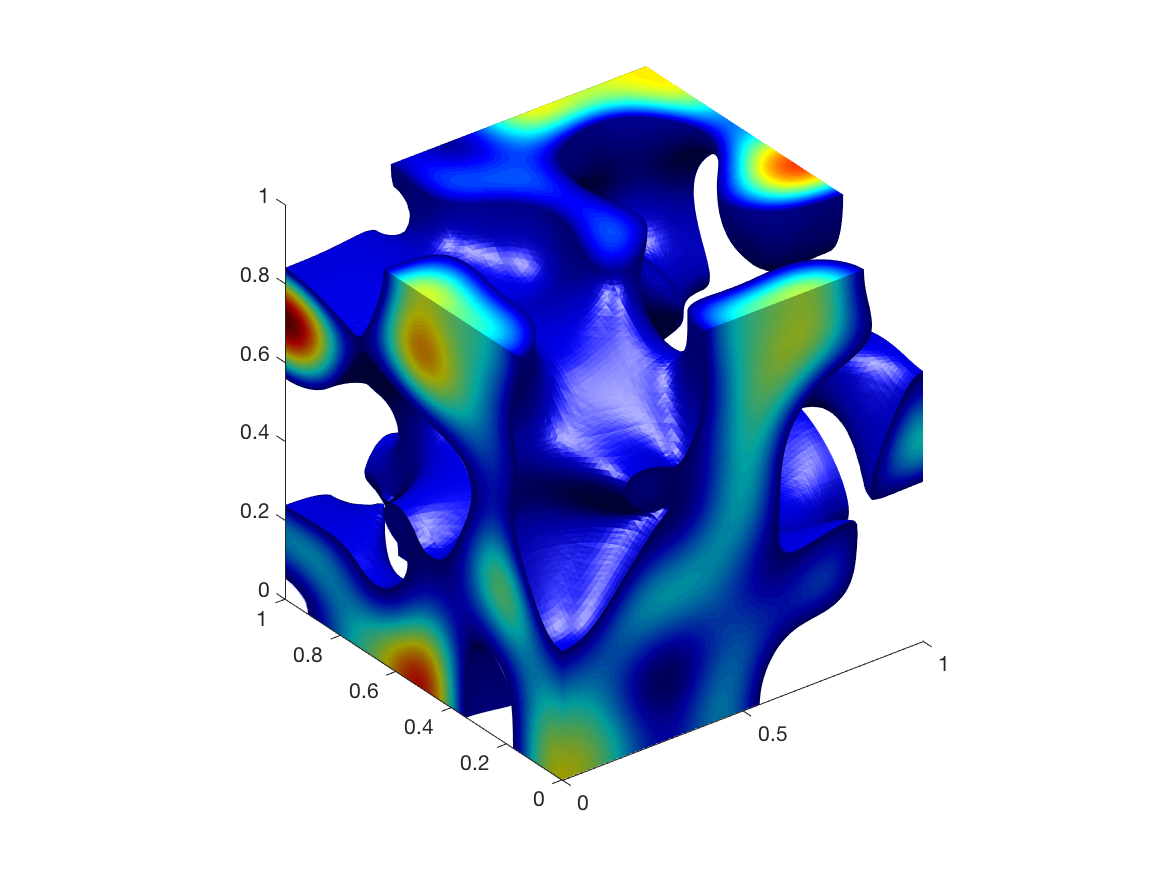}}
\subfloat{\includegraphics[width = 2in]{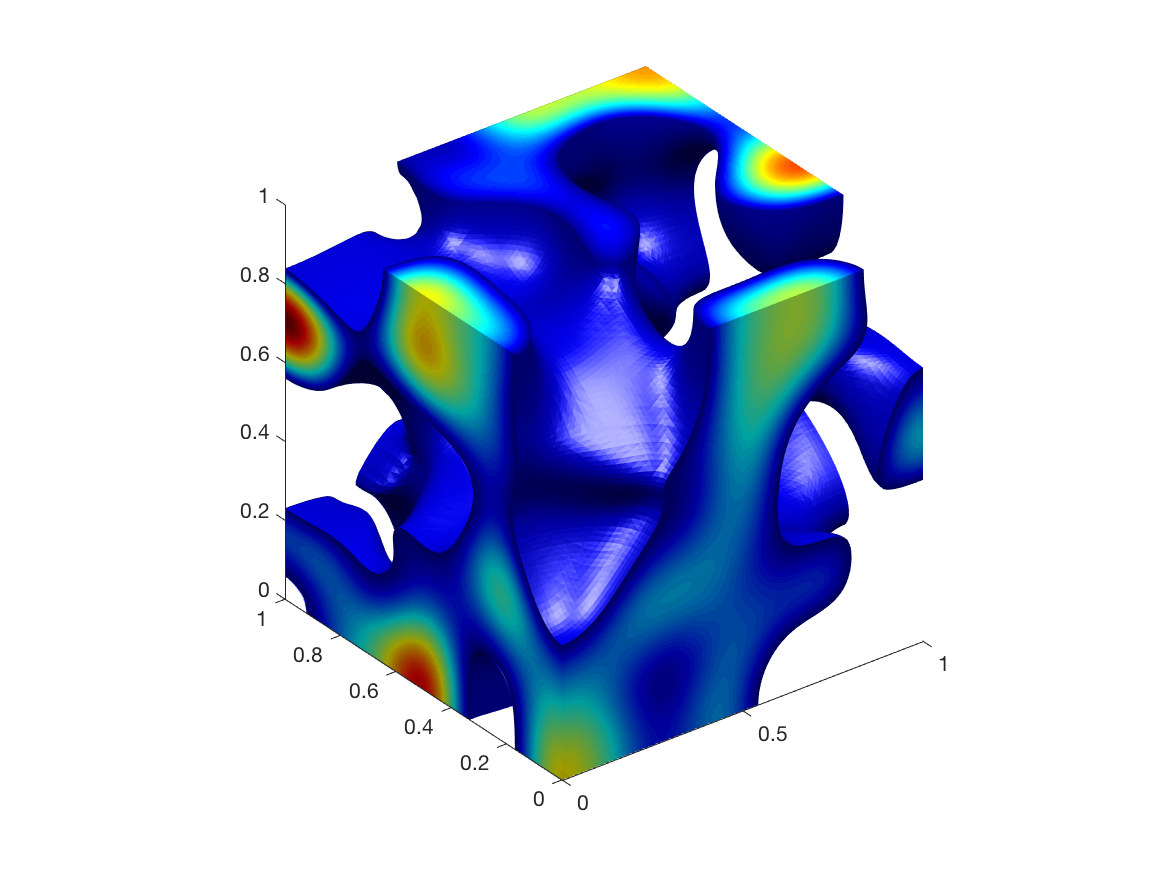}}
\subfloat{\includegraphics[width = 2in]{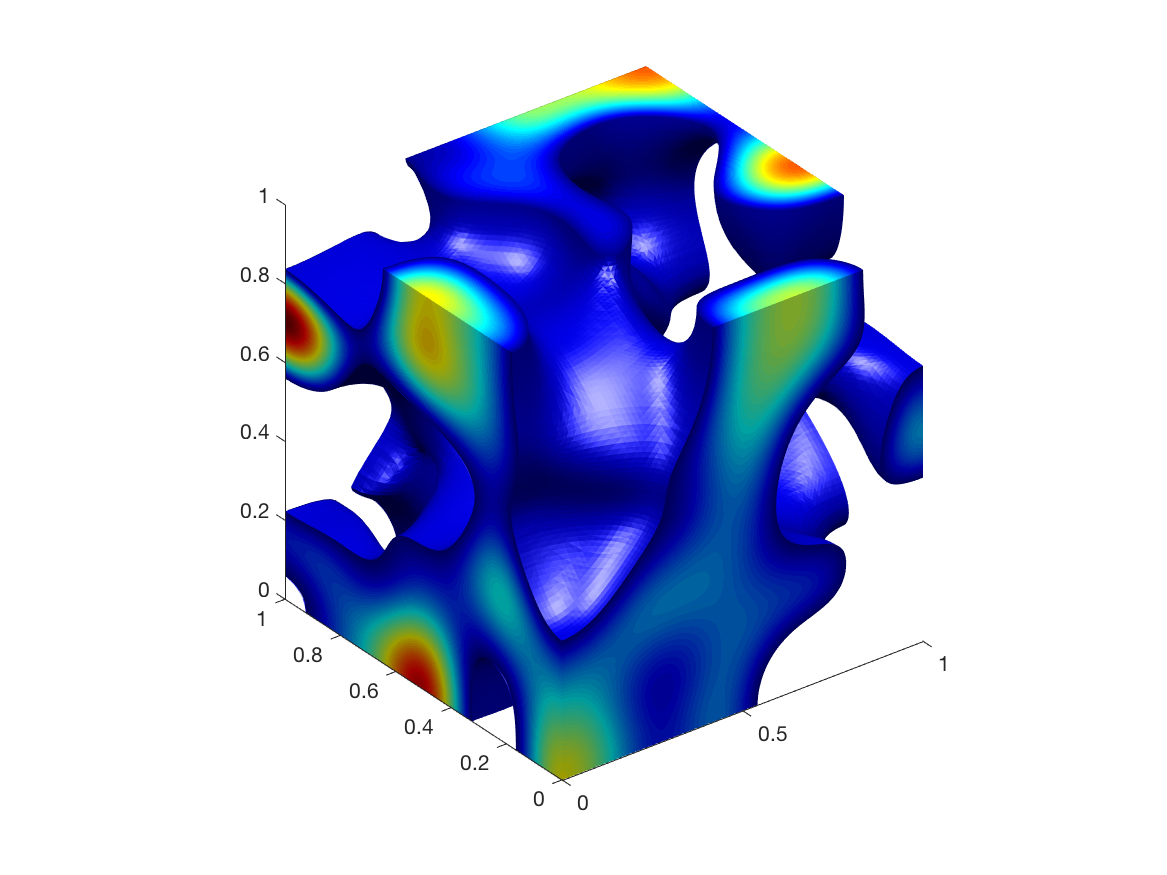}}\\
\caption{Spinodal decomposition of a binary fluid on $(0,1)^3$.
 The times displayed are $t=0, t=0.0015625, t=0.003125$ (top from left to right)
 and $t=0.0046875, t=0.00625, t=0.0078125$ (bottom from left to right).}
\label{fig:rand-matlab-3D}
\end{figure}
\par
 In the eighth experiment, we solve the Cahn-Hilliard equation with a random initial condition.
 We take $h=\nicefrac{\sqrt{3}}{32}, \tau = \nicefrac{0.002}{64}$ and $\varepsilon = 0.0625$.
 Isocap plots for $t=0, t=0.015$ and $t=0.03$ are displayed in Figure \ref{fig:3D-matlab}.
 For comparison, we solve the same problem using FEniCS and display the corresponding isocap plots in
 Figure~\ref{fig:3D-fenics}. The two figures are, again, essentially indistinguishable.
 Furthermore, we achieve considerable savings in time by using our solver.
 Specifically, the test using our solver completed in under 30 minutes whereas the test using
 FEniCS required 24 hours to reach the same final stopping time of $T=0.03$.

\begin{figure}[H]
\subfloat{\includegraphics[width = 2in]{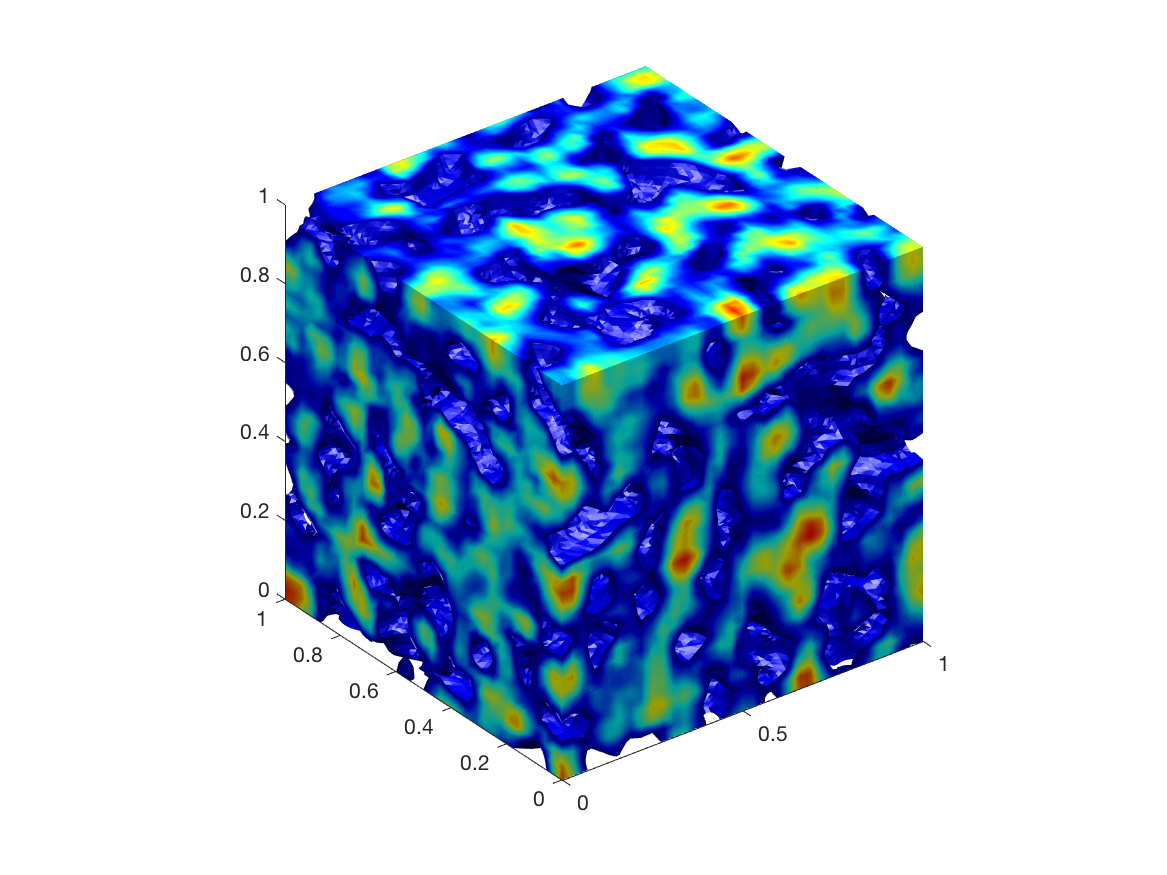}}
\subfloat{\includegraphics[width = 2in]{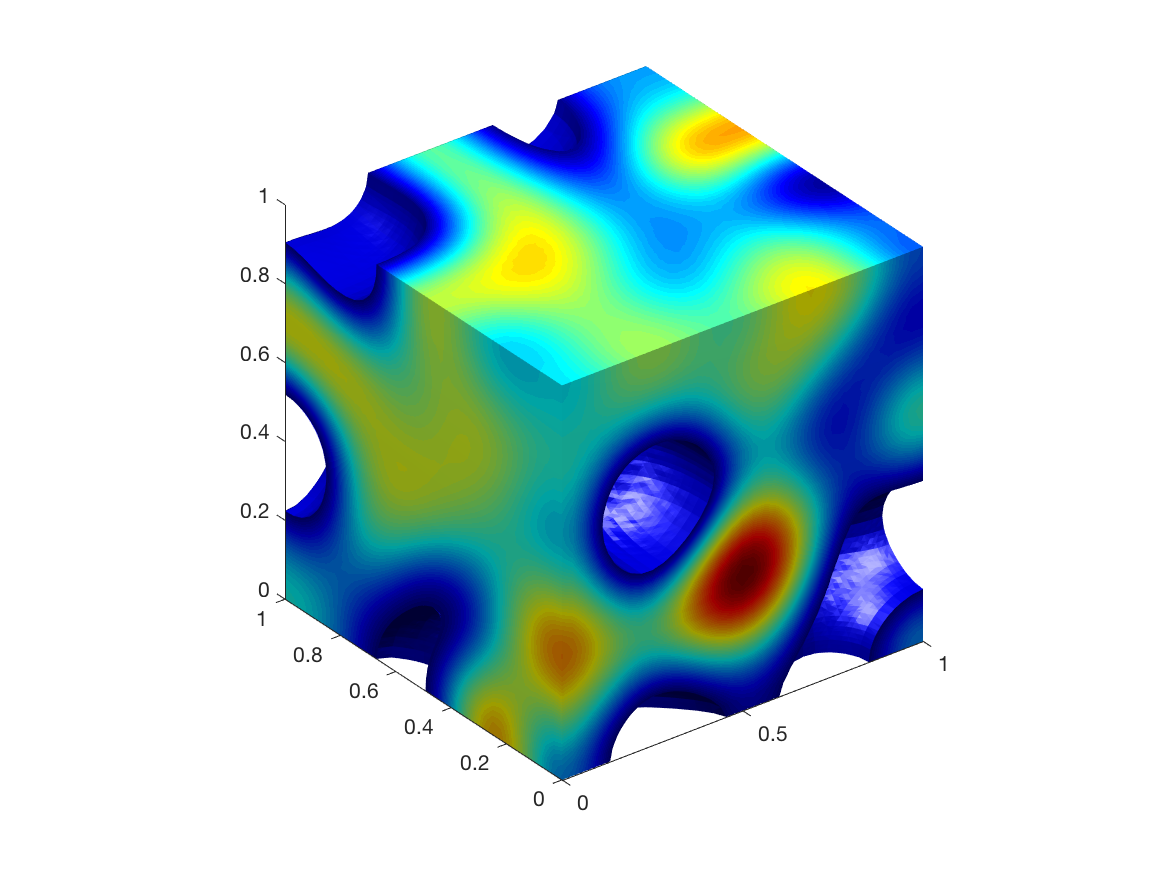}}
\subfloat{\includegraphics[width = 2in]{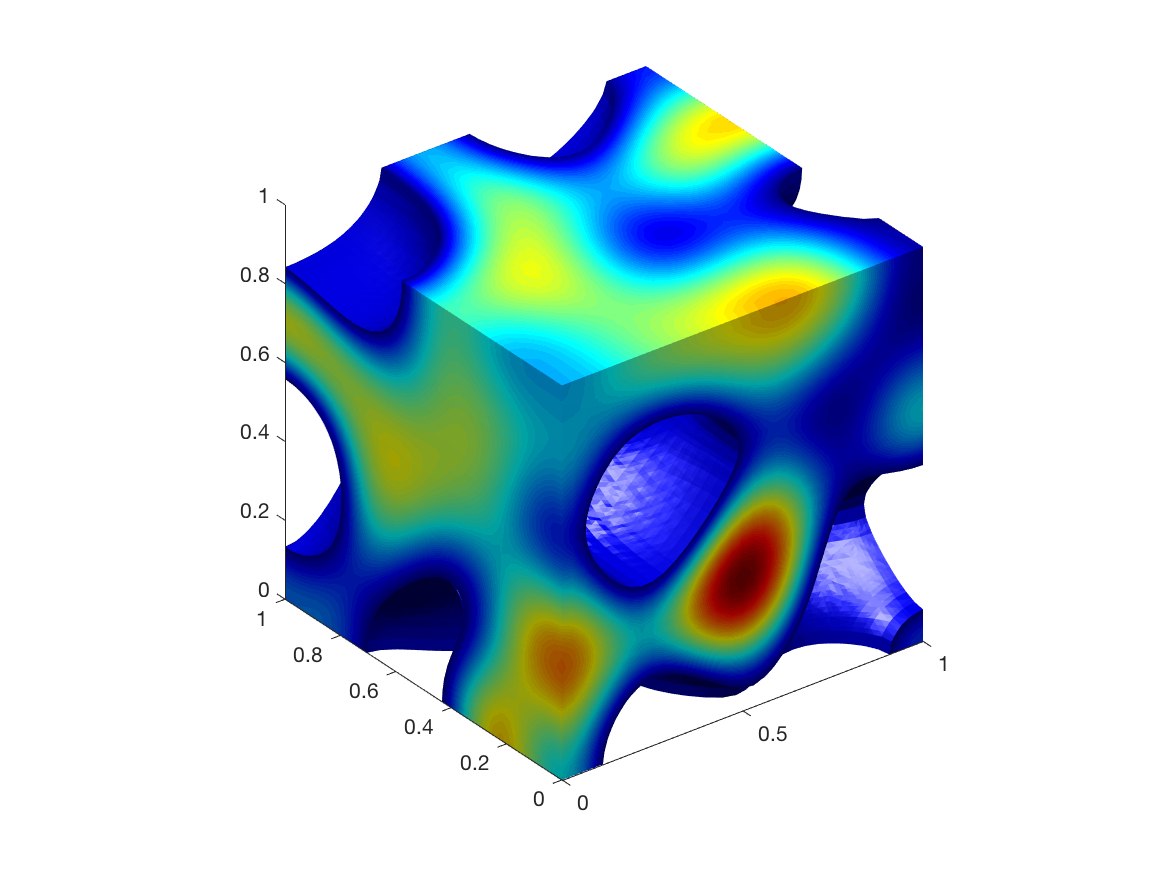}}
\caption{Spinodal decomposition of a binary fluid on $(0,1)^3$ with random initial data. The times displayed are $t=0, t=0.015, t=0.03$ (from left to right).}
\label{fig:3D-matlab}
\end{figure}

\begin{figure}[H]
\subfloat{\includegraphics[width = 2in]{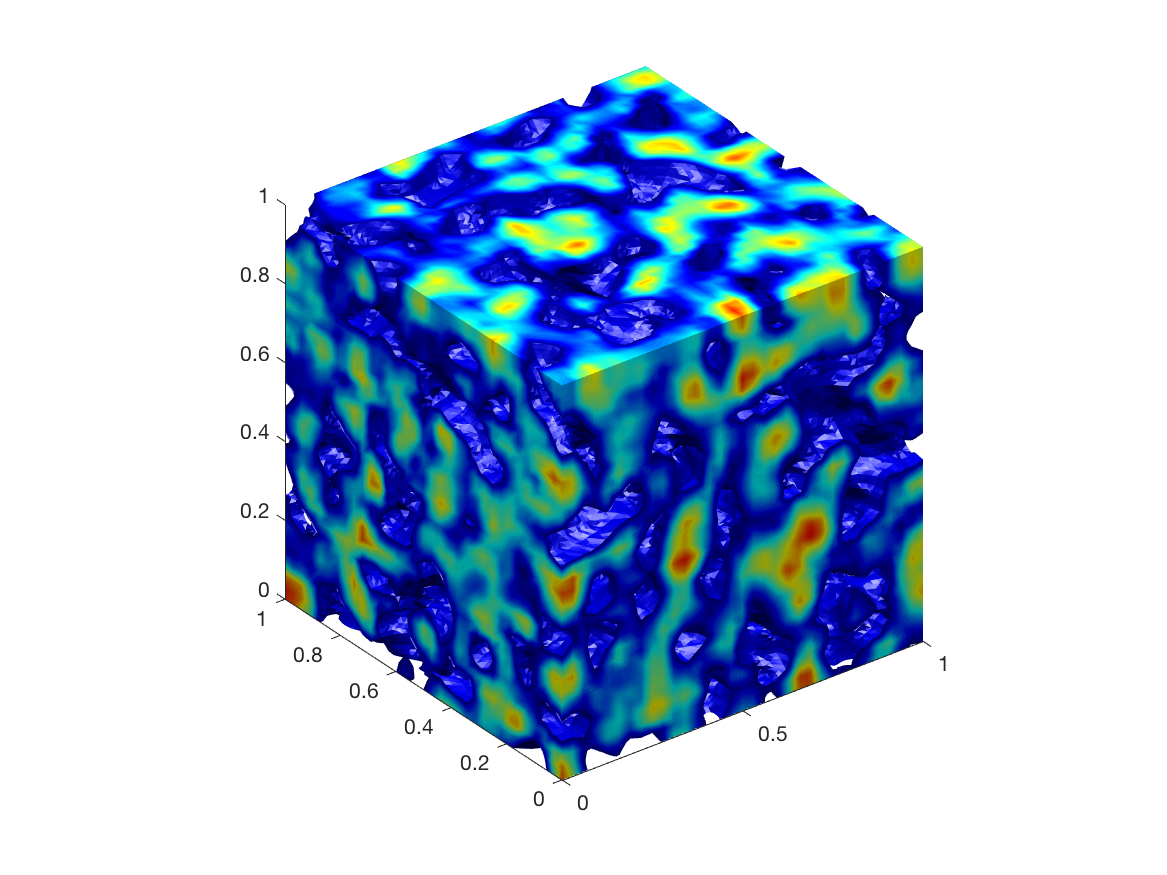}}
\subfloat{\includegraphics[width = 2in]{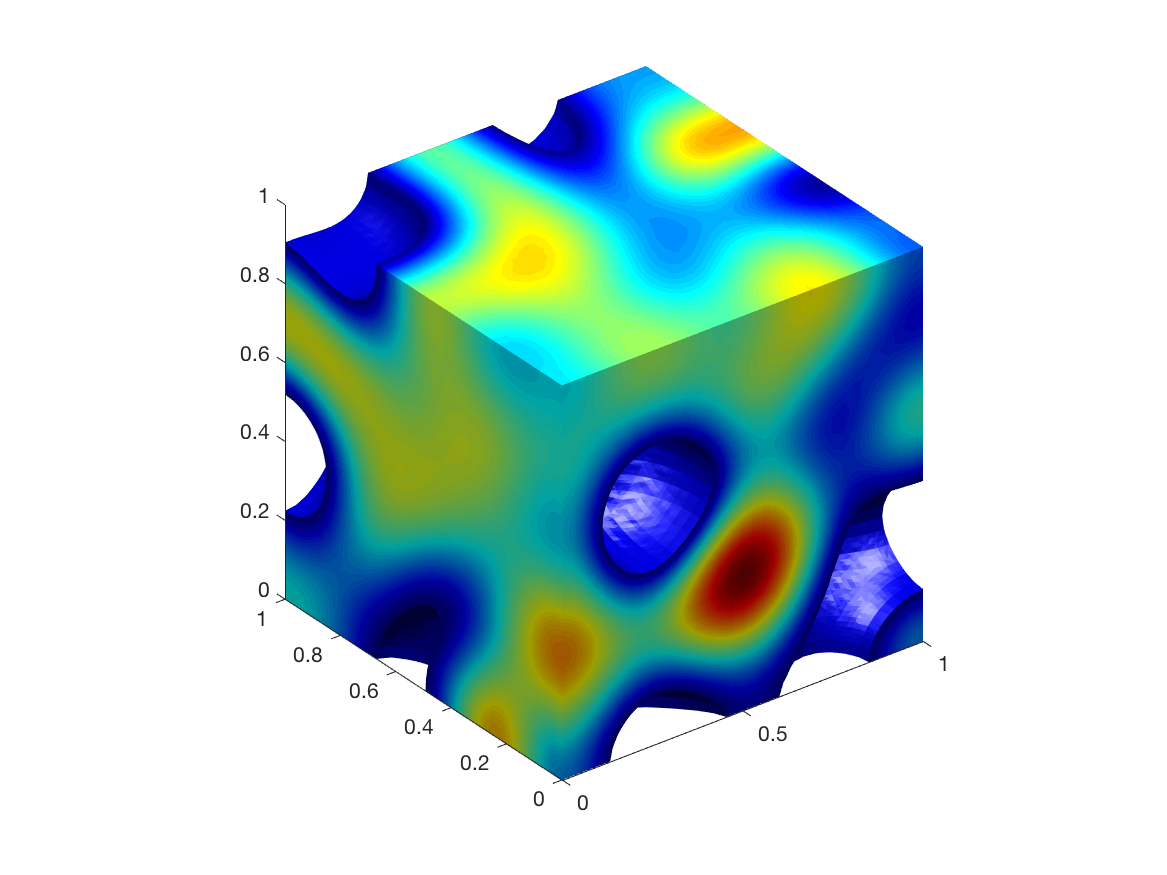}}
\subfloat{\includegraphics[width = 2in]{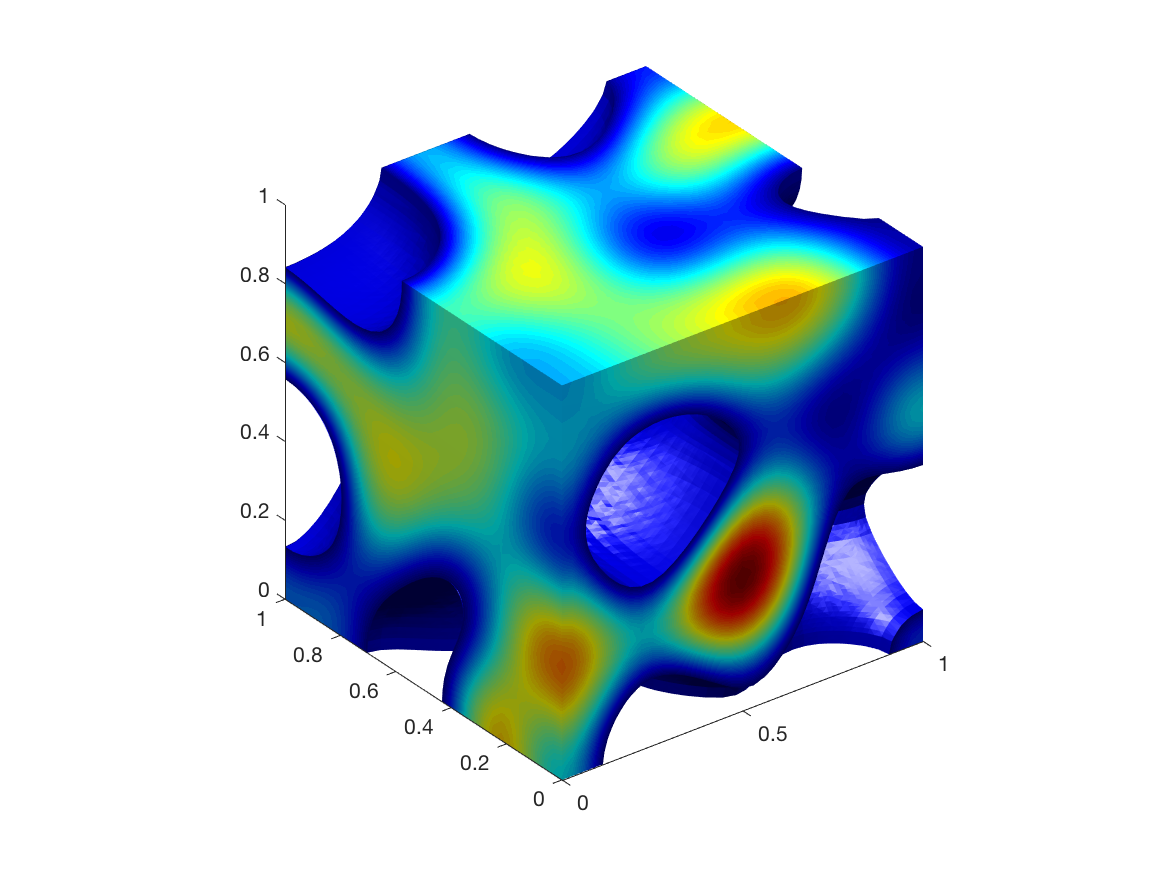}}
\caption{Spinodal decomposition of a binary fluid on $(0,1)^3$ with random initial data using
 FEniCS. The times displayed are $t=0, t=0.015, t=0.03$ (from left to right).}
\label{fig:3D-fenics}
\end{figure}

\section{Concluding Remarks}\label{sec:Conclusions}
 We have developed a robust solver for a mixed finite element convex splitting scheme for the
 Cahn-Hilliard equation,  where in each time step the Jacobian system for the Newton iteration
 is solved by a preconditioned MINRES algorithm with a block diagonal multigrid preconditioner.
 The robustness of our solver is confirmed by numerical tests in two and three dimensions.
  We have also validated our numerical results through comparisons with the results obtained through
  FEniCS and observed significant speed-up.
\par
 The methodology developed in this paper can be adapted for coupled systems that involve
 the Cahn-Hilliard equation, such as the Cahn-Hilliard Navier Stokes system
 (cf. \cite{tierra15} and the references therein).
 This is the topic for an ongoing research project.
\section*{Acknowledgement}
Portions of this research were conducted with high performance
computational resources provided by Louisiana State University
(http://www.hpc.lsu.edu).
 We would also like to thank Shawn Walker for his valuable advice regarding the
 FELICITY/C++ Toolbox for MATLAB.


\begin{thebibliography}{10}

\bibitem{fenics15}
M.~S. Aln{\ae}s, J.~Blechta, J.~Hake, A.~Johansson, B.~Kehlet, A.~Logg,
  C.~Richardson, J.~Ring, E.~Rognes M, and G.~N. Wells.
\newblock The fenics project version 1.5.
\newblock {\em Archive of Numerical Software}, 3, 2015.

\bibitem{axelsson13}
O.~Axelsson, P.~Boyanova, M.~Kronbichler, M.~Neytcheva, and X.~Wu.
\newblock Numerical and computational efficiency of solvers for two-phase
  problems.
\newblock {\em Comput. Math. Appl.}, 65:301--314, 2013.

\bibitem{barrett05}
J.W. Barrett, R.~Nurnberg, and V.~Styles.
\newblock Finite element approximation of a phase field model for void
  electromigration.
\newblock {\em SIAM J. Numer. Anal.}, 42(2):738--772, 2005.

\bibitem{BGL:2005:SaddlePoint}
M.~Benzi, G.H. Golub, and J.~Liesen.
\newblock {Numerical solution of saddle point problems}.
\newblock {\em Acta Numerica}, 14:1--137, 2005.

\bibitem{boyanova12}
P.~Boyanova, M.~Do-Quang, and M.~Neytcheva.
\newblock Efficient preconditioners for large scale binary {Cahn-Hilliard}
  models.
\newblock {\em Comput. Methods Appl. Math.}, 12(1):1--22, 2012.

\bibitem{BScott:2008:FEM}
S.C. Brenner and L.R. Scott.
\newblock {\em {The Mathematical Theory of Finite Element Methods $($Third
  Edition$)$}}.
\newblock Springer-Verlag, New York, 2008.

\bibitem{cahn61}
J.W. Cahn.
\newblock On spinodal decomposition.
\newblock {\em Acta Metall.}, 9:795, 1961.

\bibitem{cahn58}
J.W. Cahn and J.E. Hilliard.
\newblock Free energy of a nonuniform system. {I.} interfacial free energy.
\newblock {\em J. Chem. Phys.}, 28:258, 1958.

\bibitem{ceniceros07}
H.~D. Ceniceros and A.~M. Roma.
\newblock A nonstiff, adaptive mesh refinement-based method for the
  {Cahn--Hilliard} equation.
\newblock {\em J. Comput. Phys.}, 225:1849--1862, 2007.

\bibitem{choksi11}
R.~Choksi, M.~Maras, and J.~F. Williams.
\newblock 2d phase diagram for minimizers of a {Cahn-Hilliard} functional with
  long-range interactions.
\newblock {\em SIAM J. Appl. Dyn. Sys.}, 10(4):1344--1362, 2011.

\bibitem{diegel15}
A.~Diegel, X.~Feng, and S.~M. Wise.
\newblock Analysis of a mixed finite element method for a
  {Cahn-Hilliard-Darcy-Stokes} system.
\newblock {\em SIAM J. Numer. Anal.}, 53(1):127--152, 2015.

\bibitem{diegel16}
A.~Diegel, C.~Wang, and S.~M. Wise.
\newblock Stability and convergence of a second-order mixed finite element
  method for the {Cahn-Hilliard} equation.
\newblock {\em IMA J. Numer. Anal.}, 36:1867--1897, 2016.

\bibitem{du07}
Q.~Du, M.~Li, and C.~Liu.
\newblock Analysis of a phase field {Navier-Stokes} vesicle-fluid interaction
  model.
\newblock {\em Discrete Contin. Dyn. Syst. Ser. B}, 8(3):539, 2007.

\bibitem{ET:1999:Convex}
I.~Ekeland and R.~T{\'e}mam.
\newblock {\em Convex Analysis and Variational Problems}.
\newblock Classics in Applied Mathematics. Society for Industrial and Applied
  Mathematics (SIAM), Philadelphia, PA, 1999.

\bibitem{elliott86}
C.M. Elliott and S.~Zheng.
\newblock On the {Cahn-Hilliard} equation.
\newblock {\em Arch. Ration. Mech. Anal.}, 96:339--357, 1986.

\bibitem{ESW:2014:FEFIS}
H.C. Elman, D.J. Silvester, and A.J. Wathen.
\newblock {\em {Finite Elements and Fast Iterative Solvers: with Applications
  in Incompressible Fluid Dynamics}}.
\newblock Oxford University Press, Oxford, second edition, 2014.

\bibitem{eyre98}
D.~Eyre.
\newblock {Unconditionally Gradient Stable Time Marching the Cahn-Hilliard
  Equation}.
\newblock In J~W Bullard, R~Kalia, M~Stoneham, and L~Q Chen, editors, {\em
  Computational and Mathematical Models of Microstructural Evolution},
  volume~53, pages 1686--1712, Warrendale, PA, USA, 1998. Materials Research
  Society.

\bibitem{feng06}
X.~Feng.
\newblock Fully discrete finite element approximations of the
  {Navier-Stokes-Cahn-Hilliard} diffuse interface model for two-phase fluid
  flows.
\newblock {\em SIAM J. Numer. Anal.}, 44:1049--1072, 2006.

\bibitem{Greenbaum:1997:Iterative}
A.~Greenbaum.
\newblock {\em {Iterative Methods for Solving Linear Systems}}.
\newblock SIAM, Philadelphia, 1997.

\bibitem{Hackbusch:1985:MMA}
W.~Hackbusch.
\newblock {\em {Multi-grid Methods and Applications}}.
\newblock Springer-Verlag, Berlin-Heidelberg-New York-Tokyo, 1985.

\bibitem{kay06}
D.~Kay and R.~Welford.
\newblock A multigrid finite element solver for the {Cahn--Hilliard} equation.
\newblock {\em J. Comput. Phys.}, 212:288--304, 2006.

\bibitem{kim07}
J.~Kim.
\newblock A numerical method for the {Cahn--Hilliard} equation with a variable
  mobility.
\newblock {\em Commun. Nonlinear Sci. Numer. Simul.}, 12:1560--1571, 2007.

\bibitem{kim04}
J.~Kim, K.~Kang, and J.~Lowengrub.
\newblock Conservative multigrid methods for {Cahn-Hilliard} fluids.
\newblock {\em J. Comput. Phys.}, 193:511--543, 2004.

\bibitem{lee14}
C.~Lee, D.~Jeong, J.~Shin, Y.~Li, and J.~Kim.
\newblock A fourth-order spatial accurate and practically stable compact scheme
  for the {Cahn--Hilliard} equation.
\newblock {\em Physica A}, 409:17--28, 2014.

\bibitem{lee14b}
D.~Lee, J.~Huh, D.~Jeong, J.~Shin, A.~Yun, and J.~Kim.
\newblock Physical, mathematical, and numerical derivations of the
  {Cahn--Hilliard} equation.
\newblock {\em Comp Mater Sci}, 81:216--225, 2014.

\bibitem{lee02a}
H.~G. Lee, J.~S. Lowengrub, and J.~Goodman.
\newblock Modeling pinchoff and reconnection in a {Hele-Shaw} cell. {I.} the
  models and their calibration.
\newblock {\em Phys. Fluids}, 14:492--513, 2002.

\bibitem{MW:2011:SaddlePoint}
K.-A. Mardal and R.~Winther.
\newblock Preconditioning discretizations of systems of partial differential
  equations.
\newblock {\em Numer. Linear Algebra Appl.}, 18:1--40, 2011.

\bibitem{shin13}
J.~Shin, S.~Kim, D.~Lee, and J.~Kim.
\newblock A parallel multigrid method of the {Cahn--Hilliard} equation.
\newblock {\em Comp Mater Sci}, 71:89--96, 2013.

\bibitem{stogner08}
R.~H. Stogner, G.~F. Carey, and B.~T. Murray.
\newblock Approximation of {Cahn--Hilliard} diffuse interface models using
  parallel adaptive mesh refinement and coarsening with {C1} elements.
\newblock {\em Internat. J. Numer. Methods Engrg.}, 76:636--661, 2008.

\bibitem{Temam:1988:DynamicalSystem}
R.~Temam.
\newblock {\em {Infinite-Dimensional Dynamical Systems in Mechanics and
  Physics}}.
\newblock Springer-Verlag, New York, 1988.

\bibitem{tierra15}
G~Tierra and F~Guill{\'e}n-Gonz{\'a}lez.
\newblock Numerical methods for solving the {Cahn--Hilliard} equation and its
  applicability to related energy-based models.
\newblock {\em Arch Comput Method E}, 22:269--289, 2015.

\bibitem{TOS:2001:MG}
U.~Trottenberg, C.~Oosterlee, and A.~Sch{\"u}ller.
\newblock {\em {Multigrid}}.
\newblock Academic Press, San Diego, 2001.

\bibitem{van09}
S.~van Teeffelen, R.~Backofen, A.~Voigt, and H.~L{\"o}wen.
\newblock Derivation of the phase-field-crystal model for colloidal
  solidification.
\newblock {\em Phys. Rev. E}, 79:051404, 2009.

\bibitem{felicity}
S.W. Walker.
\newblock {FELICITY}: {F}inite {EL}ement {I}mplementation and {C}omputational
  {I}nterface {T}ool for {Y}ou.
\newblock http://www.mathworks.com/matlabcentral/fileexchange/31141-felicity.

\bibitem{wang16}
W.~Wang, L.~Chen, and J.~Zhou.
\newblock Postprocessing mixed finite element methods for solving
  {Cahn--Hilliard} equation: Methods and error analysis.
\newblock {\em J Sci Comput}, 67:724--746, 2016.

\bibitem{wise10}
S.M. Wise.
\newblock Unconditionally stable finite difference, nonlinear multigrid
  simulation of the {Cahn-Hilliard-Hele-Shaw} system of equations.
\newblock {\em J. Sci. Comput.}, 44:38--68, 2010.

\bibitem{zhou15}
J.~Zhou, L.~Chen, Y.~Huang, and W.~Wang.
\newblock An efficient two-grid scheme for the {Cahn-Hilliard} equation.
\newblock {\em Commun Comput Phys}, 17:127--145, 2015.

\end{thebibliography}
	\end{document}